\documentclass[smallextended]{svjour3}       
\usepackage{graphicx}

\usepackage{amsmath,amssymb}
\usepackage[mathscr]{eucal}

\usepackage{mathtools}                                                      
\mathtoolsset{showonlyrefs=true}                                    

\let\counterwithin\relax
\usepackage{chngcntr}

\usepackage{ulem}

\usepackage{dsfont}

\usepackage{amsmath}                                                        
\usepackage{graphicx}                                                       
\usepackage{subfigure}
\usepackage{enumitem}                                                    
\usepackage{hyperref}
\usepackage{amssymb}                                                        
\usepackage[mathscr]{eucal}                                             
\usepackage{cancel}                                                             
\usepackage{pstricks}
\usepackage{rotating}
\usepackage{lscape}
\usepackage[paperwidth=8.5in,paperheight=11in,top=1.25in, bottom=1.25in, left=1.00in, right=1.00in]{geometry}
\usepackage{mathtools}                                                      
\mathtoolsset{showonlyrefs=true}                                    
\usepackage{fixltx2e,amsmath}                                           
\MakeRobust{\eqref}
\allowdisplaybreaks                                                             
\usepackage{chngcntr}

\usepackage{ulem}       
\usepackage{mathabx}

\usepackage{titlesec}
\usepackage[titletoc,toc,title]{appendix}

\usepackage{apptools, etoolbox}

\AtAppendix{\counterwithin{theorem}{section}}

\usepackage{color}

\newcommand{\cev}[1]{\,\reflectbox{\ensuremath{\vec{\reflectbox{\!\ensuremath{#1}}}}}}

\newtheorem{assumption}[theorem]{Assumption}

\def \Qg {\mathbf{Y}}
\def \Qgb {\cev{\mathbf{Y}}}

\def \gvar {\gamma}
\def \bvar {\cev{\gamma}}
\def \gY {\gamma^{\mathbf{B}}}
\def \bIW {\cev{\gamma}^{\text{\rm\tiny IW}}}

\def \gIW {\gamma^{\text{\rm\tiny IW}}}

\def \FIW {g^{\text{\rm\tiny IW}}}
\def \FIWi {g^{\text{\rm\tiny IW},-1}}
\def \BIW {\cev{g}^{\text{\rm\tiny IW}}}
\def \BIWi {\cev{g}^{\text{\rm\tiny IW},-1}}
\def \au {\bar{a}}
\def \bu {\bar{b}}

\def \MM {\mathbf{m}}

\def \o {{\omega}}
\def \a {{\alpha}}
\def \b {{\beta}}
\def \d {{\delta}}
\def \l {{\lambda}}

\def \pG {\mathbf{\Gamma}}
\def \G {{\Gamma}}
\def \s {{\sigma}}
\def \w {{\omega}}

\def \R {{\mathbb {R}}}
\def \N {{\mathbb {N}}}

\def \x {{\xi}}
\def \e {{\varepsilon}}

\def \r {{\varrho}}

\def \t {{\tau}}
\def \t {{\tau}}
\def \n {{\nu}}
\def \m {{\mu}}
\def \y {{\eta}}

\def \z {{\zeta}}

\def \g {{\gamma}}
\def \O {{\Omega}}
\def \phi {{\varphi}}

\def \tilde {\widetilde}

\def\p{\partial}

\def \B {{\cal{B}}}

\def \o {{\omega}}
\def \a {{\alpha}}
\def \b {{\beta}}
\def \d {{\delta}}
\def \l {{\lambda}}

\def \G {{\Gamma}}
\def \s {{\sigma}}
\def \w {{\omega}}

\def \ww {z,y}

\def \fLL {\tilde{\mathcal{L}}}
\def \fLLL {{\mathcal{L}}}
\def \fG {\mathcal{G}}
\def \fGG {\tilde{\mathcal{G}}}
\def \fA {\mathcal{A}}
\def \fAA {\tilde{\mathcal{A}}}
\def \fB {\tilde{B}}

\def \R  {{\mathbb {R}}}
\def \x {{\xi}}
\def \g {{\gamma}}
\def \e {{\varepsilon}}

\def \t {{\tau}}
\def \n {{\nu}}
\def \m {{\mu}}
\def \y {{\eta}}

\def \z {{\zeta}}
\def \p {{\partial}}
\def \a {{\alpha}}
\def \O {{\Omega}}

\def \d {{\delta}}

\def \o {{\omega}}
\def \a {{\alpha}}
\def \b {{\beta}}
\def \d {{\delta}}

\def \G {\Ga}

\def \Ga {{\Gamma}}

\def \s {{\sigma}}

\def \w {{\omega}}
\def \R {{\mathbb {R}}}
\def \N {{\mathbb {N}}}

\def \x {{\xi}}
\def \e {{\varepsilon}}

\def \r {{\varrho}}

\def \t {{\tau}}
\def \t {{\tau}}
\def \n {{\nu}}
\def \v {{\nu}}
\def \m {{\mu}}
\def \y {{\eta}}

\def \z {{\zeta}}

\def \g {{\gamma}}
\def \O {{\Omega}}
\def \phi {{\varphi}}

\def \tilde {\widetilde}

\def\l {\lambda}

\def \A {\mathcal{A}}

\def \F {\mathcal{F}}
\def \B {\mathscr{B}}

\def \à {{\`a }}
\def \è {{\`e }}
\def \ò {{\`o }}
\def \ù {{\`u }}

\def \bV {b}
\def \sV {\sigma}
\def \sVi {\hat{\sigma}}
\def \sVii {\sigma^{1}}
\def \sY {\theta}
\def \bY {h}
\def \hY {\tilde{h}}

\def \fLL {\tilde{\mathcal{L}}}
\def \fG {\mathcal{G}}
\def \fGG {\tilde{\mathcal{G}}}
\def \fA {\mathcal{A}}
\def \fAA {\tilde{\mathcal{A}}}

\numberwithin{equation}{section}

\begin{document}

\title{Backward and forward filtering under the weak H\"ormander condition
}


\author{Andrea Pascucci         \and
        Antonello Pesce 
}


\institute{A. Pascucci \at
              Department of Mathematics, University of Bologna, Piazza di Porta san Donato, 5 Bologna \\
              Tel.: +39-0512094428\\
              \email{andrea.pascucci@unibo.it}           
           \and
           A. Pesce \at
              Department of Mathematics, University of Bologna, Piazza di Porta san Donato, 5 Bologna \\
              \email{antonello.pesce2@unibo.it}           
}

\date{Received: date / Accepted: date}

\maketitle

\begin{abstract}
We derive the forward and backward filtering equations for a class of degenerate partially
observable diffusions, satisfying the weak H\"ormander condition. Our approach is based on the
H\"older theory for degenerate SPDEs that allows to pursue the direct approaches proposed by N. V.
Krylov and A. Zatezalo, and A. Yu. Veretennikov, avoiding the use of general results from
filtering theory. As a by-product we also provide existence, regularity and estimates for the
filtering density. 
\end{abstract}

\section{Introduction}
The classical kinetic model
\begin{equation}\label{aaee1}
  \begin{cases}
    dX_{t}=V_{t}dt, \\
    dV_{t}=\s dW_{t},\qquad \s>0,
  \end{cases}
\end{equation}
is a remarkable example of a system of SDEs whose Kolmogorov equation
\begin{equation}\label{aaee1bb}
  \frac{\s^{2}}{2}\p_{vv}f+v\p_{x}f+\p_{t}f=0,\qquad (t,x,v)\in\R^{3},
\end{equation}
is hypoelliptic but not uniformly parabolic. Precisely, \eqref{aaee1bb} satisfies the {\it weak}
H\"ormander condition in that the drift plays a key role in the noise propagation (see
\cite{Kolmogorov2} and the introduction in \cite{Hormander}). In \eqref{aaee1} $W$ is a Brownian
motion and $X,V$ represent position and velocity of a particle. This type of SDEs arises in
several linear and non-linear models in physics (see, for instance, \cite{Cercignani},
\cite{Lions1},
\cite{Desvillettes}, \cite{MR2130405}) 
and in mathematical finance (see, for instance, \cite{BarucciPolidoroVespri},
\cite{Pascucci2011}).

In this paper we study the filtering problem for \eqref{aaee1}. To the best of our knowledge, this
kind of problem was never considered in the literature, possibly because the known results for
hypoelliptic SPDEs (e.g. \cite{MR736147}, \cite{MR705933}, \cite{Krylov17}, \cite{MR3839316} and
\cite{MR3706782}) do not apply in this case. Here we propose a unified approach for the derivation
of the {\it backward and forward filtering equations} based on the H\"older theory for degenerate
SPDEs recently developed in \cite{PascucciPesce1} and \cite{pasc:pesc:19} (see also \cite{Chow94}
and \cite{MR1755998} for similar results for uniformly parabolic SPDEs).
Having an existence and regularity theory at hand, we can pursue the ``direct'' approaches
proposed by Krylov and Zatezalo \cite{MR1795614} and Veretennikov \cite{Veretennikov}, thus
avoiding the use of general results from filtering theory. In particular, as in \cite{Veretennikov} we
derive the backward filtering equation ``by hand'', 
without resorting to prior knowledge of the SPDE, in a more direct way compared to the classical
approach in \cite{MR553909}, \cite{MR583435}, \cite{MR1070361} or \cite{MR3839316}.

To be more specific, we consider the following general setup: we assume that the position $X_{t}$
and the velocity $V_{t}$ of a particle are scalar stochastic processes only partially observable
through some observation process $Y_{t}$. The joint dynamics of $X,V$ and $Y$ is given by the
system of SDEs
\begin{equation}\label{eq1}
\begin{cases}
  dX_t=V_tdt,\\
  dV_t={\bV(t,X_t,V_t,Y_t)dt}+
  {\bar\sV(t,X_t,V_t,Y_t)dW^1_t+\hat\sV(t,X_t,V_t,Y_t)dW^2_t},\\
  dY_t=\bY(t,X_t,V_t,Y_t)dt+
  \sY(t,Y_t)dW^1_t,
\end{cases}
\end{equation}
where {$W_t=(W_t^1,W_t^{2})$ denotes a bi-dimensional Brownian motion} defined on a complete
probability space $(\O,\F,P)$ with a filtration $(\F_t)_{t\in [0,T]}$ satisfying the usual
assumptions. Hereafter, for simplicity we set $Z_{t}=(X_{t},V_{t})$ and denote by $z=(x,v)$ and
$\z=(\x,\n)$ the points in $\R^{2}$.

Let $\F_{t,T}^{Y}=\s(Y_s,t\le s\le T)$ define the filtration of observations and let $\phi$ be a
bounded and continuous function, $\phi\in bC(\R^{2})$. The filtering problem consists in finding
the best $\F_{t,T}^{Y}$-measurable least-square estimate of $\phi(Z_{T})$, that is the conditional
expectation $E\left[\phi(Z_{T})\mid \F_{t,T}^{Y}\right]$. Our first result, Theorem \ref{th2}
shows that
\begin{equation}\label{ae7bis}
 E\left[\phi(Z_{T}^{t,z})\mid {\F^{Y}_{t,T}}\right]=
 \int_{\R^{2}}\hat{\mathbf{\pG}}(t,z;T,\z)\phi(\z)d\z, 
\end{equation}
where $\hat{\mathbf{\pG}}$ is the (normalized) fundamental solution of the {\it forward filtering
equation}; the latter is a SPDE of the form
\begin{equation}\label{spde_forwbbb}
 d_{\mathbf{B}}u_s(\z)=\A_{s,\z}u_s(\z)ds+\mathcal{G}_{s,\z}u_s(\z){dW^1_s}
\end{equation}
where $\mathbf{B}=\p_s+\n\p_\x$ and 
\begin{align}
 \A_{s,\z}u_s(\z)&=\frac{1}{2}{\left(\bar\sV^{2}+\hat\sV^{2}\right)(s,\z,Y_s)}\p_{\n\n}u_s(\z)+\text{\it ``lower order terms''},\\ 
  {\mathcal{G}_{s,\z}u_s(\z)}&{=\bar{\sV}(s,\z,Y_s)\p_{\n}u_s(\z)+\text{\it ``lower order terms''}}.
\end{align}
The forward filtering SPDE is precisely formulated in \eqref{Forward_eq2}. The symbol
$d_{\mathbf{B}}$ in \eqref{spde_forwbbb} indicates that the SPDE is understood in the It\^o (or
strong) sense, that is
\begin{equation}\label{spde_forw1bbb}
 u_{s}\left(\gY_{s-t}(\z)\right)=u_{t}(\z)+\int_{t}^{s}{\A_{\t,\gY_{\t-t}(\z)}}u_\t(\gY_{\t-t}(\z))d\t
 + \int_{t}^{s}\mathcal{G}_{\t,\gY_{\t-t}(\z)}u_\t(\gY_{\t-t}(\z)){dW^1_{\t}},\qquad s\in[t,T],
\end{equation}
where $s\mapsto\gY_{s}(\x,\n)$ denotes the integral curve, starting from $(\x,\n)$, of the
advection vector field $\n\p_{\x}$ or, more explicitly, $\gY_{s}(\x,\n)=(\x+s\n,\n)$.
{\begin{example}
The prototype of \eqref{spde_forwbbb} is the Langevin SPDE
\begin{equation}\label{spde_forw_pro}
 d_{\mathbf{B}}u_s(\x,\n)=\frac{\s^{2}}{2}\p_{\n\n}u_s(\x,\n)ds+\b\p_{\n}u_s(\x,\n){dW^1_s},
\end{equation}
with $\s,\b$ constant parameters. Clearly, if $u_s=u_s(\x,\n)$ is a smooth function then
\eqref{spde_forw_pro} can be written in the usual It\^o form
\begin{equation}\label{spde_forw_pro_bis}
 du_s(\x,\n)=\left(\frac{\s^{2}}{2}\p_{\n\n}u_s(\x,\n){-}\n\p_{\x}u_s(\x,\n)\right)ds+\b\p_{\n}u_s(\x,\n){dW^1_s}.
\end{equation}
Notice that $\p_{\x}$, being equal to the Lie bracket $[\p_{\n},\mathbf{B}]$, has to be regarded
as a {\it third order derivative} in the intrinsic sense of subelliptic operators (cf.
\cite{MR657581}): this motivates the use of the ``Lie stochastic differential'' $d_{\mathbf{B}}$
instead of the standard It\^o differential in \eqref{spde_forw}. Notice also that
\eqref{spde_forw_pro} reduces to the forward Kolmogorov (or Fokker-Planck) equation for
\eqref{aaee1} when $\b=0$.
\end{example}

Analogously, in Section \ref{bSPDE} we prove that
\begin{equation}
 E\left[\phi(Z_{T}^{t,z,y},Y_{T}^{t,z,y})\mid {\F^{Y}_{t,T}}\right]=
 \int\limits_{\R^{3}}\bar{\pG}(t,z,y;T,\z,\y)\phi(\z,\y)d\z d\y, \qquad (t,z,y)\in [0,T]\times
 \R^2\times\R,
\end{equation}
where $\bar{\pG}$ denotes the (normalized) fundamental solution of the {\it backward filtering
equation} that is a SPDE of the form
\begin{equation}\label{spde_backbbb}
 -d_{{\mathbf{B}}}u_t(z,y)=\fAA_{t}
 u_t(z,y)dt+\fGG_{t}u_t(z,y)\star {dW^1_t}.
\end{equation}
We refer to \eqref{Backward_eq2} for the precise formulation of the backward filtering SPDE. The
symbol $\star$ means that \eqref{spde_backbbb} is written in terms of the {\it backward It\^o
integral} whose definition is recalled in Section \ref{Itoback} for reader's convenience.  We
shall see that the coefficients of the forward filtering SPDE are random, while the coefficients
of the backward filtering SPDE are deterministic. Moreover, \eqref{spde_forwbbb} is posed in
$\R^{3}$ (including the time variable) while \eqref{spde_backbbb} is posed in $\R^{4}$.

The rest of the paper is organized as follows. In Section \ref{sec1} we resume and extend the
H\"older theory for degenerate SPDEs satisfying the weak H\"ormander condition, developed in
\cite{PascucciPesce1} and \cite{pasc:pesc:19}. In Section \ref{sec2}, which is the core of the
paper, we state the filtering problem and derive the forward and backward filtering SPDEs. Section
\ref{proofK} contains the proof of the results about the existence and Gaussian estimates for the
fundamental solutions of the filtering SPDEs. In Section \ref{Itoback} we review the definition
and some basic result about backward stochastic integration. For reader's convenience, in Section
\ref{secnot} we collect the main notations systematically used throughout the paper.

\section{Fundamental solution of Langevin-type SPDEs}\label{sec1}
We present the H\"older theory for degenerate SPDEs that will be used in the derivation of the
filtering equations.
{Compared to \cite{pasc:pesc:19}, here we state our results in a slightly more general setting
where the dimension of the non degenerate variable $v$ can be possibly greater than one. This is
done with the purpose of being able to handle differential operators constructed from the generator
of the full process $(X_t,V_t,Y_t)$ which will appear in Section \ref{sec2}, both in the derivation of
the forward and the backward filtering SPDE, and it is not related to the number $n$ of Brownian motions
considered in the model \eqref{eq1}. On the other hand the dimension lift does not bring any additional
difficulty in the analysis since it is performed in the non degenerate directions.}

We first introduce some general notation and the functional spaces used
throughout the paper.

We denote by $z=(x,v_{1},\dots,v_{d})$ and $\z=(\x,\n_{1},\dots,\n_{d})$ the points in
$\R\times\R^d$. Moreover, for any $k\in\N$, $0<\a<1$ and $0\le t<T$,
\begin{itemize}
  \item[i)] $m\B_{t,T}$  (resp. $b\B_{t,T}$) is the space of all real-valued (resp. bounded)
  Borel measurable functions $f=f_{s}(z)$ on $[t,T]\times \R^{d+1}$;
  \item[ii)] $C^{0}_{t,T}$ (resp. $bC^{0}_{t,T}$) is the space of functions $f\in m\B_{t,T}$ (resp. $f\in b\B_{t,T}$)
  that are continuous in $z$ and $C^{\a}_{t,T}$ (resp. $bC^{\a}_{t,T}$) is the space of functions $f\in m\B_{t,T}$ (resp. $f\in b\B_{t,T}$)
  that are $\a$-H\"older continuous in $z$ uniformly with respect to $s$, that is
  $$\sup_{s\in[t,T]\atop z\neq \z}\frac{|f_{s}(z)-f_{s}(\z)|}{|z-\z|^{\a}}<\infty.$$
  {We also denote by  $C^{0,1}_{t,T}$ the space of functions $f\in m\B_{t,T}$ that are Lipschitz continuous in $z$ uniformly with respect to
  $s\in [t,T]$};
  \item[iii)] $C^{k+\a}_{t,T}$ (resp. $bC^{k+\a}_{t,T}$) is the space of functions $f\in m\B_{t,T}$ that are $k$-times differentiable with respect to $z$
  with derivatives in $C^{\a}_{t,T}$  (resp. $bC^{\a}_{t,T}$).
\end{itemize}
We use boldface to denote the stochastic version of the previous functional spaces. Let
$(W_t)_{t\in [0,T]}$ be a one-dimensional Brownian motion on a complete probability space
$(\O,\F,P)$, endowed with a filtration $(\F_t)_{t\in [0,T]}$ satisfying the usual conditions, and
let $\mathcal{P}_{t,T}$ be the predictable $\s$-algebra on $[t,T]\times\O$.
\begin{definition}\label{def1}
We denote by $\mathbf{C}^{k+\a}_{t,T}$ the family of functions $f=f_{s}(z,\o)$ on
$[t,T]\times\R^{d+1}\times\O$ such that:
\begin{itemize}
 \item[i)] $(s,z)\mapsto f_{s}(z,\o)\in C^{k+\a}_{t,T}$ for any $\o\in\O$;
 \item[ii)] $(s,\w)\mapsto f_{s}(z,\o)$ is $\mathcal{P}_{t,T}$-measurable for any $z\in\R^{d+1}$.
\end{itemize}
Similarly, we define $\mathbf{bC}^{k+\a}_{t,T}$.
\end{definition}

We consider a class of degenerate SPDEs of the form
\begin{equation}\label{spde_forw}
 d_{\mathbf{B}}u_s(\z)=\A_{s,\z}u_s(\z)ds+\mathcal{G}_{s,\z}u_s(\z)dW_s
\end{equation}
where $\mathbf{B}=\p_s+\n_1\p_\x$ and 
\begin{align}
 \A_{s,\z}u_s(\z)&:=\frac{1}{2}a^{ij}_s(\z)\p_{\n_i\n_j}u_s(\z)+b^i_s(\z)\p_{\n_i}u_s(\z)+c_s(\z)u_s(\z), \\
 \mathcal{G}_{s,\z}u_s(\z)&:=\s^{i}_s(\z)\p_{\n_i}u_s(\z)+h_s(\z)u_s(\z).
\end{align}
\begin{definition}\label{ad1}
A solution to \eqref{spde_forw} on $[t,T]$ is a process $u=u_{s}(\x,\n)\in \mathbf{C}^{0}_{t,T}$
that is twice continuously differentiable in the variables $\n$ and solves the equation
\begin{equation}\label{spde_forw1}
 u_{s}\left(\gY_{s-t}(\z)\right)=u_{t}(\z)+\int_{t}^{s}{\A_{\t,\gY_{\t-t}(\z)}}u_\t(\gY_{\t-t}(\z))d\t
 + \int_{t}^{s}\mathcal{G}_{\t,\gY_{\t-t}(\z)}u_\t(\gY_{\t-t}(\z))dW_{\t},\qquad s\in[t,T],
\end{equation}
where $s\mapsto\gY_{s}(\x,\n)$ denotes the integral curve, starting from $(\x,\n)$, of the
advection vector field $\n_1\p_{\x}$, that is $\gY_{s}(\x,\n)=(\x+s\n_1,\n)$.
\end{definition}
\begin{definition}\label{d1}
A fundamental solution of the forward SPDE \eqref{spde_forw} is a stochastic process
${\pG}={\pG}(t,z;s,\z)$, defined for $0\le t<s\le T$ and $z,\z\in\R^{d+1}$, such that for any
$(t,z)\in[0,T)\times\R^{d+1}$ and $t_{0}\in(t,T)$ we have:
\begin{itemize}
  \item[i)] ${\pG}(t,z;\cdot,\cdot)$ is a solution to \eqref{spde_forw} on $[t_{0},T]$;
  \item[ii)] for any $\phi\in bC(\R^{d+1})$ and $z_{0}\in\R^{d+1}$, we have 
    $$\lim_{(s,\z)\to(t,z_{0})\atop s>t}\int_{\R^{2}}
    {\pG}(t,z;s,\z)\phi(z)dz=\phi(z_{0}),\qquad P\text{-a.s.}$$
\end{itemize}
\end{definition}
In \cite{pasc:pesc:19}, under suitable assumptions on the coefficients, we proved existence and
Gaussian-type estimates of a fundamental solution for \eqref{spde_forw} when $b_s\equiv c_s\equiv
h_s\equiv 0$ and $d=1$. Here we slightly extend those results to an SPDE of the general form
\eqref{spde_forw} and to the backward version of it, that is
\begin{equation}\label{spde_back}
 -d_{\mathbf{B}}u_t(z)=\A_{t,z}u_t(z)dt+\mathcal{G}_{t,z}u_t(z)\star dW_t, \qquad
 \mathbf{B}=\p_t+v_1\p_x.
\end{equation}

We denote by $\cev{\mathbf{C}}^{k+\a}_{t,T}$ (and $\mathbf{b}\cev{\mathbf{C}}^{k+\a}_{t,T}$) the
stochastic H\"older spaces formally defined as in Definition \ref{def1} with $\mathcal{P}_{t,T}$
in condition ii) replaced by the backward predictable $\s$-algebra $\cev{\mathcal{P}}_{t,T}$
defined in terms of the backward Brownian filtration (cf. Section \ref{Itoback}). Again,
\eqref{spde_back} is understood in the strong sense:
\begin{definition}\label{ad2}
A solution to \eqref{spde_back} on $[0,s]$ is a process $u=u_{t}(x,v)\in
\cev{\mathbf{C}}^{0}_{0,s}$ that is twice continuously differentiable in the variables $v$ and
such that
\begin{equation}\label{spde_back1}
 u_{t}\left(\gY_{s-t}(z)\right)=u_{s}(z)+\int_{t}^{s}{\A_{\t,\gY_{s-\t}(z)}}u_\t(\gY_{s-\t}(z))d\t
 + \int_{t}^{s}\mathcal{G}_{\t,\gY_{s-\t}(z)}u_\t(\gY_{s-\t}(z))\star dW_{\t},\qquad t\in[0,s].
\end{equation}
\end{definition}
%
%
\begin{definition}\label{d2} A fundamental solution for
the backward SPDE \eqref{spde_back} is a stochastic process $\cev{\pG}=\cev{\pG}(t,z;s,\z)$
defined for $0\le t<s\le T$ and $z,\z\in\R^{d+1}$, such that for any $(s,\z)\in
(0,T]\times\R^{d+1}$ and $t_{0}\in(0,s)$ we have:
\begin{itemize}
  \item[i)] $\cev{\pG}(\cdot,\cdot;s,\z)$ is a solution to \eqref{spde_back} on $[0,t_{0}]$;
 \item[ii)] for any $\phi\in bC(\R^{d+1})$ and $z_{0}\in\R^{d+1}$, we have
    $$\lim_{(t,z)\to(s,z_{0})\atop t<s}\int_{\R^{2}}
    \cev{\pG}(t,z;s,\z)\phi(\z)d\z=\phi(z_{0}),\qquad P\text{-a.s.}$$
\end{itemize}
\end{definition}

Next we pose the standing assumptions on the coefficients of \eqref{spde_forw} and
\eqref{spde_back}.
\begin{assumption}[\bf Regularity]\label{ass1}
For some $\a\in(0,1)$, we have:
\begin{itemize}
  \item[i)] $a\in\mathbf{bC}^{\a}_{0,T}$, $\s\in\mathbf{bC}^{3+\a}_{0,T}$, $b,c\in\mathbf{bC}^{0}_{0,T}$, $h\in
\mathbf{bC}^2_{0,T}$ in the forward SPDE \eqref{spde_forw};
  \item[ii)] $a\in\mathbf{b}\,\cev{\mathbf{C}}^{\a}_{0,T}$, $\s\in\mathbf{b}\,\cev{\mathbf{C}}^{3+\a}_{0,T}$,
$b,c\in\mathbf{b}\,\cev{\mathbf{C}}^{0}_{0,T}$, $h\in \mathbf{b}\,\cev{\mathbf{C}}^2_{0,T}$ in the
backward SPDE \eqref{spde_back}.
\end{itemize}
\end{assumption}
\begin{assumption}[\bf Coercivity]\label{ass2}
There exists a random, finite and positive constant $\MM$ such that
\begin{equation}
 \langle (a_{t}(z)-\s_{t}(z)\s_t^{\ast}(z))\z,\z\rangle \ge \MM |\z|^2, \qquad t\in [0,T],\ z,\z\in \R^{d+1},\ P\text{-a.s.}
\end{equation}
\end{assumption}
In our analysis we make use of the It\^o-Wentzell transform. Let $(x,v)\in\R^{d+1}$. For a fixed
$t\in [0,T)$ we consider the SDE in $\R^{d}$
\begin{align}\label{sde_forw}
 \gIW_{t,s}(x,v)&=v-\int_{t}^s\s_\t(x,\gIW_{t,\t}(x,v))dW_\t, \qquad s\in [t,T],
\intertext{and, for a fixed $s\in (0,T]$, the SDE}
 \bIW_{t,s}(x,v)&=v+\int_{t}^{s}\s_\t(x,\bIW_{\t,s}(x,v))\star dW_\t,\qquad t\in[0,s].\label{sde_back}
\end{align}
Assumption \ref{ass1} ensures that \eqref{sde_forw} and \eqref{sde_back} have strong solutions and
the maps $(x,v)\mapsto \left(x,\gIW_{t,s}(x,v)\right)$ and $(x,v)\mapsto
\left(x,\bIW_{t,s}(x,v)\right)$ define forward and backward flows of diffeomorphisms of $\R^{d+1}$
respectively. These changes of coordinates allow to transform the SPDEs \eqref{spde_forw} and
\eqref{sde_back} into PDEs with random coefficients whose properties depend on the gradient of the
stochastic flow: to have a control on it, we impose the some additional condition. {For any
suitably regular function $f=f(w):\R^{N}\longrightarrow \R$, $\e>0$ and multi-index $\b\in
\N_{0}^{N}$, we set
\begin{equation}\label{aea1}
 \langle f\rangle_{\e,\b}:=\sup_{w\in\R^{N}}(1+|w|^2)^{\e}|\p_{w}^{\b}f(w)|.
\end{equation}
\begin{assumption}\label{ass3}
There exist $\e>0$ and two random variables $M_1\in L^{p}(\O)$, with
$p>\max\left\{2,\frac{1}{\e}\right\}$, and $M_2\in L^{\infty}(\O)$ such that with probability one
\begin{align}
  \sup_{t\in[0,T]}\left(\langle \s_{t}\rangle_{\e,\b}+\langle \s_{t}\rangle_{1/2+\e,\b'}\right)&\le M_{1},\qquad |\b|=1,\
  |\b'|=2,3,\\
  \sup_{t\in[0,T]}\langle h_{t}\rangle_{1/2,\b}&\le M_{2},\qquad |\b|=1.
\end{align}
\end{assumption}} Assumption \ref{ass3} requires that $\s_{t}(z)$ and $h_t(z)$ flatten as $z\to \infty$. In
particular, this condition is clearly satisfied if $\s$ and $h$ depend only on $t$ or, more
generally, if the spatial gradients of $\s$ and $h$ have compact support.

In order to state the main result of this section, Theorem \ref{TH1} below, we need to introduce
some additional notation: we consider the Gaussian kernel
\begin{equation}\label{gammal}
 \Gamma_{\l}(t,x,v)=\frac{1}{t^{\frac{d+3}{2}}}\exp\left(-\frac{1}{2\l}\left(\frac{x^{2}}{t^{3}}+\frac{|v|^{2}}{t}\right)\right),
  \qquad t>0,\ (x,v)\in\R\times\R^{d},\ \l>0.
\end{equation}
To fix ideas, for $d=1$ and up to some renormalization, $\Gamma_{\l}$ is the fundamental solution
of the degenerate Langevin equation \eqref{aaee1bb}. For a recent survey on the theory of this
kind of ultra-parabolic operators and the related sub-elliptic structure, we refer to
\cite{anceschi}.

In the following statement, we denote by $\FIWi$ (and $\BIWi$) the inverse of the It\^o-Wentzell
stochastic flow $(x,v)\mapsto \FIW(x,v):= \left(x,\gIW_{t,s}(x,v)\right)$ defined by
\eqref{sde_forw} (and $(x,v)\mapsto \BIW(x,v):= \left(x,\bIW_{t,s}(x,v)\right)$ defined by
\eqref{sde_back}, respectively). Moreover, we consider the vector field
\begin{align}\label{flusso2}
  \Qg_{t,s}(z)&:=\Big((\gIW_{t,s})_1(z),-(\gIW_{t,s}(z))_1(\nabla_v\gIW_{t,s})^{-1}(z){\p_x\gIW_{t,s}(z)}\Big),
\end{align}
with $\nabla_v\gIW=(\p_{v_j}\gIW_i)_{i,j=1,\cdots d}$ and $\p_x\gIW=(\p_x\gIW_i)_{i=1,\cdots d}$,
and define $\Qgb_{t,s}$ analogously, {namely
 $$\Qgb_{t,s}(z):=\Big((\bIW_{t,s})_1(z),-(\bIW_{t,s}(z))_1(\nabla_v\bIW_{t,s})^{-1}(z){\p_x\bIW_{t,s}(z)}\Big).$$}
 Eventually, equation
  $$\gvar_{s}^{t,z}=z+\int_{t}^{s}\Qg_{t,\t}(\gvar_{\t}^{t,z})d\t,\qquad s\in[t,T],$$
defines the integral curve of $\Qg_{t,s}$ starting from $(t,z)$, and equation
  $$\bvar_{t}^{s,\z}=\z+\int_{t}^{s}\Qgb_{\t,s}(\bvar_{\t}^{s,\z})d\t,\qquad t\in[0,s],$$
defines the integral curve of $\Qgb_{t,s}$ ending at $(s,\z)$. The main result of this section is
the following theorem whose proof is postponed to Section \ref{proofK}; for reader's convenience,
in Section \ref{secnot} we collect the main notations used hereafter.
\begin{theorem}\label{TH1}
Under Assumptions \ref{ass1}-i), \ref{ass2} and \ref{ass3}, the forward SPDE \eqref{spde_forw} has
a fundamental solution $\mathbf{\pG}$ and there exist two positive random variables $\lambda$,
$\mu$ such that
\begin{align}
\mu^{-1}\Gamma_{\l^{-1}}\left({s-t}, \FIWi_{t,s}(\z)-\gvar_{s}^{t,z} \right) &\leq
 \mathbf{\pG}(t,z;s,\z) \leq \mu\Gamma_{\l}\left({s-t}, \FIWi_{t,s}(\z)-\gvar_{s}^{t,z}\right), \label{t_e1}\\
 \left| \p_{\n_i}\mathbf{\pG}(t,z;s,\x,\n)\right|&\leq
 \frac{\mu}{\sqrt{s-t}}\Gamma_{\l}\left({s-t}, \FIWi_{t,s}(\x,\n)-\gvar_{s}^{t,z}\right), \label{t_e2}\\
 \left| \p_{\n_i\n_j}\mathbf{\pG}(t,z;s,\x,\n)\right|
 &\leq \frac{\mu}{s-t}\Gamma_{\l}\left({s-t}, \FIWi_{t,s}(\x,\n)-\gvar_{s}^{t,z}\right), \label{t_e3}
\end{align}
for every $0\leq t<s\leq T$, $z,\z=(\x,\n)\in \R^{d+1}$ and $i,j=1,\cdots d$, with probability
one.

Similarly, under Assumptions \ref{ass1}-ii), \ref{ass2} and \ref{ass3}, the backward SPDE
\eqref{spde_back} has a fundamental solution $\cev{\pG}$ satisfying estimates
\begin{align}
 \mu^{-1}\Gamma_{\l^{-1}}\left({s-t}, \BIWi_{t,s}(z)-\bvar_{t}^{s,\z}\right)
 &\leq \cev{\pG}(t,z;s,\z) \leq \mu\Gamma_{\l}\left({s-t}, \BIWi_{t,s}(z)-\bvar_{t}^{s,\z}\right), \label{t_e4}\\
 \left| \p_{v_i}\cev{\pG}(t,x,v;s,\z)\right|&\leq
 \frac{\mu}{\sqrt{s-t}}\Gamma_{\l}\left({s-t}, \BIWi_{t,s}(x,v)-\bvar_{t}^{s,\z}\right), \label{t_e5}\\
 \left| \p_{v_iv_j}\cev{\pG}(t,x,v;s,\z)\right|
 &\leq \frac{\mu}{s-t}\Gamma_{\l}\left({s-t}, \BIWi_{t,s}(x,v)-\bvar_{t}^{s,\z}\right), \label{t_e6}
\end{align}
for every $0\leq t<s\leq T$, $z=(x,v),\z\in \R^{d+1}$ and $i,j=1,\dots d$, with probability one.
\end{theorem}
\begin{remark}
We would like to emphasize that Theorem \ref{TH1} is new even in the deterministic case, i.e. when
$\s\equiv 0$, $h\equiv 0$ and the coefficients are deterministic functions. In fact, a study of
Kolmogorov PDEs with coefficients measurable in time was only recently proposed in \cite{brampol}:
however in \cite{brampol} the coefficients are assumed to be independent of the spatial variables
that is a very particular case where the fundamental solution is known explicitly.
\end{remark}
In the derivation of the forward filtering SPDE, we will use a deterministic backward Kolmogorov
PDE to which Theorem \ref{TH1} applies. Precisely, we will use the following
\begin{corollary}\label{L_Cauchy}
Let Assumption \ref{ass2} with $\s\equiv 0$ be satisfied and let $a\in bC^{\a}_{0,T}$, $b,c\in
bC^{0}_{0,T}$, for some $\a\in(0,1)$, and $\phi\in bC(\R^{d+1})$. Then there exists a bounded
solution of the backward Cauchy problem
\begin{equation}\label{e_Cauchy2}
  \begin{cases}
   -d_{\mathbf{B}}f(t,z)=\A_{t,z}f(t,z)dt,\\
   f(T,\cdot)=\phi,
  \end{cases}
\end{equation}
in the sense of Definition \ref{ad2}, that is
\begin{equation}\label{e_Cauchy}
 f\left(t,\gY_{T-t}(z) \right)=\phi(z)+\int_t^T \A_{s,\gY_{T-s}(z)} f\left(s,\gY_{T-s}(z)\right)ds, \qquad
 (t,z)\in[0,T]\times \R^{d+1},
\end{equation}
where $\gY_{s}(x,v)=(x+sv_1,v)$. Moreover, {if $\phi\in bC^{\a}(\R^{d+1})$ for some $\a\in (0,1)$}
then there exists a positive constant $C$ such that,
\begin{equation}\label{coreq}
 {\sup_{(x,v)\in \R\times
 \R^{d}}|\p_{v}^{\b}f(t,x,v)|\le C(T-t)^{-\frac{|\b|-\a}{2}},\qquad 1\le |\b|\le 2.}
\end{equation}
\end{corollary}

\section{The filtering problem}\label{sec2}
Consider system \eqref{eq1} and suppose that $h\equiv \sY\equiv \bar\sV\equiv 0$, that is no observation is available on 
the solution $Z^{t,z}$ starting from $z$ at time $t$. Then, it is well known that, under suitable
regularity and non-degeneracy assumptions on
$\hat\sV$, we have 
\begin{equation}\label{aaee10}
 E\left[\phi(Z_{T}^{t,z})\right]=\int_{\R^2}\G(t,z;T,\z)\phi(\z)d\z,
\end{equation}
where 
$\G=\G(t,z;T,\z)$ is the fundamental solution of the backward Kolmogorov 
operator 
\begin{equation}\label{eai1}
 \mathcal{K}=\frac{\hat{\sV}^{2}}{2}\p_{vv}+b\p_{v}+v\p_x+\p_t,
\end{equation}
with respect to the variables $(t,x,v)$ and of its adjoint, the Fokker-Plank operator
$\mathcal{K}^{\ast}$,
w.r.t the forward variables $(T,\x,\n)$.

In this section we study the filtering problem for system \eqref{eq1} and, assuming that $Y$ is
not trivial, we prove a representation formula for $E\left[\phi(Z_{T})\mid \F_{t,T}^{Y}\right]$
that is analogous to \eqref{aaee10} in the sense that it is written in terms of the fundamental
solution of a backward and a forward SPDE, whose existence is guaranteed by Theorem \ref{TH1}.
Actually, in filtering theory, the derivation itself of the filtering SPDE is a non-trivial task.

As already mentioned, in our analysis we will adopt a {\it direct approach}. However, we should
acknowledge that there are at least two quite different {\it direct approaches} proposed in the
literature: both of them are meant to avoid the purely probabilistic techniques of the general
filtering theory.

The direct approach by Krylov and Zatezalo \cite{MR1795614} mimics the derivation of the standard
Kolmogorov operator \eqref{eai1}: roughly speaking, {\it assuming that the filtering SPDE is known
in advance}, one takes a solution $u_{t}$ (whose existence is guaranteed by Theorem \ref{TH1}),
applies the It\^o formula to $u_{t}(Z_{t})$ and finally takes expectations. This is the approach
we follow in Section \ref{fSPDE} to prove the existence of the {\it forward filtering density} and
the representation of the conditional expectation $E\left[\phi(Z_{T})\mid \F_{t,T}^{Y}\right]$ in
terms of it.

On the other hand, the direct approach by Veretennikov \cite{MR1352749}, \cite{Veretennikov},
allows to derive the backward filtering SPDE ``by hand'', without knowing the equation in advance:
the main tools are the backward It\^o calculus and the remarkable {\it backward diffusion SPDE} of
Theorem \ref{t1}. We follow this approach in Section \ref{bSPDE} to derive the backward filtering
SPDE and the corresponding {\it filtering density}. Note however that in Section \ref{bSPDE} we
only provide an informal, yet quite detailed, derivation: a full proof is outside the scope of the
present paper and would require a generalization of the results of Section \ref{Itoback} to
degenerate diffusions. This is certainly possible but would require some additional effort and is
postponed to future research.

We notice that {system \eqref{eq1}} can be written more conveniently as 
\begin{equation}\label{eq2}
\begin{cases}
 dZ_t=B Z_t dt +\mathbf{e}_2\left(\bV(t,Z_t,Y_t)dt+\bar\sV(t,Z_t,Y_t)dW^1_t+\hat\sV(t,Z_t,Y_t)dW^2_t\right),\\
 dY_t=\bY(t,Z_t,Y_t)dt+\sY(t,Y_t)dW^1_t,
\end{cases}
\end{equation}
with
 $$B=\begin{pmatrix}
   0 & 1 \\
   0 & 0 \
 \end{pmatrix},\qquad \mathbf{e}_2=\begin{pmatrix} 0 \\ 1
\end{pmatrix}.$$ 
{Hereafter we assume the following non-degeneracy condition:}
{\begin{assumption}[Coercivity]\label{A1bis} There exists a positive constant $m$ such that
\begin{align}
 \sY(t,y)\ge m, \qquad \sVi(t,z,y)\ge m, \qquad t\in [0,T], \ z\in\R^2, \ y\in\R.
\end{align}
\end{assumption}}
%

\subsection{Forward filtering SPDE}\label{fSPDE}
We consider the solution $(Z^{t,z}_{s},Y_{s})_{s\in[t,T]}$ of system \eqref{eq2} with initial
condition $Z^{t,z}_{t}=z\in\R^{2}$; we do not impose any initial condition on the $Y$-component.
We set {$\sV=(\bar\sV,\hat\sV)$} and introduce the stochastic processes
$${\begin{aligned}
  \sV_{s}(\z)&:=\sV(s,\z,Y_{s}),\qquad  
  \sY_s:=\sY(s,Y_{s}),\qquad
  \bV_s(\z):=\bV_s(\z,Y_s),\qquad \tilde{h}_s(\z):=\frac{\bY(s,\z,Y_{s})}{\sY(s,Y_{s})},
\end{aligned}}$$
The {\it forward filtering SPDE} for system \eqref{eq2} reads as follows
\begin{equation}\label{Forward_eq2}
 d_{\mathbf{B}}v_s(\x,\n)=\fA^{\ast}_{s}v_s(\x,\n)ds+\fG^{\ast}_sv_s(\x,\n)\frac{dY_s}{\sY_s},
 \qquad \mathbf{B}=\p_s+\n\p_\x,
\end{equation}
where $\fA^{\ast}$ and $\fG^{\ast}$ are the adjoints of the differential operators (with random
coefficients)
  $$\fA_s:={\frac{|\sV_{s}(\x,\n)|^{2}}{2}\p_{\n\n}}+\bV_s(\x,\n)\p_\n,\qquad \fG_{s}:={\bar\sV_s(\x,\n)}\p_{\n}+\tilde{h}_s(\x,\n),$$
respectively.

In order to apply to \eqref{Forward_eq2} the general results of Section 1, in particular Theorem
\ref{TH1} and Corollary \ref{L_Cauchy}, we assume the following conditions. 
We recall notation \eqref{aea1}.
\begin{assumption}[Regularity]\label{A2} The coefficients of \eqref{eq2} are such that
${\bar{\sV}\in bC^{3+\a}_{0,T}(\R^3)}$, $\sVi\in bC^{2+\a}_{0,T}(\R^3)$, $\sY\in
bC^{\a}_{0,T}(\R)$, $b\in bC^{1}_{0,T}(\R^3)$, $h\in bC^{2}_{0,T}(\R^3)$.
\end{assumption}
{\begin{assumption}[Flattening at infinity]\label{A3} There exist two positive constants $\e,M$
such that
\begin{align}
 {\sup_{t\in[0,T]\atop y\in\R}\left(\langle \bar\sV(t,\cdot,y)\rangle_{\e,\b}+\langle
 \bar\sV(t,\cdot,y)\rangle_{1/2+\e,\b'} +\langle h(t,\cdot,y)\rangle_{1/2,\b}\right)}&\le M
\end{align}
for $|\b|=1$ and $|\b'|=2,3$.
\end{assumption}}


\begin{remark}\label{ar1}
With regard to the existence of solutions to \eqref{Forward_eq2}, let us introduce the process
 $$\tilde{W}_{s}:=\int_t^{s}\sY_\t^{-1}dY_\t={W^{1}_{s}-W^{1}_{t}}+\int_t^{s}\hY_\t(Z^{t,z}_{\t})d\t, \qquad s\in [t,T].$$
By Girsanov's theorem, $(\tilde{W}_{s})_{s\in[t,T]}$ is a Brownian motion w.r.t the measure $Q$
defined by $dQ=(\r^{t,z}_T)^{-1}dP$ where
\begin{equation}\label{ae11}
 d\r^{t,z}_s=\hY_{s}(Z^{t,z}_{s})^2\r^{t,z}_sdt+\hY_{s}(Z^{t,z}_{s})\r^{t,z}_s{dW^{1}_s},\qquad \r^{t,z}_t=1.
\end{equation}
Moreover, $(\tilde{W}_{s})_{s\in[t,T]}$ is adapted to $(\F^{Y}_{t,s})_{s\in[t,T]}$.
Then, equation \eqref{Forward_eq2} can be written in the equivalent form
\begin{equation}\label{Forward_eq3}
 d_{\mathbf{B}}v_s(\z)=\fA^{\ast}_{s}v_s(\z)ds+\fG^{\ast}_sv_s(\z){d\tilde{W}_s}
\end{equation}
under $Q$. Under Assumptions \ref{A1bis}, \ref{A2} and \ref{A3}, by Theorem \ref{TH1} a
fundamental solution $\mathbf{\pG}=\mathbf{\pG}(t,z;s,\z)$ for \eqref{Forward_eq3} exists,
satisfies estimates \eqref{t_e1}, \eqref{t_e2}, \eqref{t_e3} and $s\mapsto
\mathbf{\pG}(t,z;s,\z)$ is adapted to $(\F^{Y}_{t,s})_{s\in[t,T]}$. 
We say that the stochastic process
\begin{equation}\label{ae100}
  \hat{\mathbf{\pG}}(t,z;s,\z)=\frac{\mathbf{\pG}(t,z;s,\z)}{\int_{\R^{2}}\mathbf{\pG}(t,z;s,\z_{1})d\z_{1}},\qquad
  0\le t<s\le T,\ z,\z\in\R^{2},
\end{equation}
is the {\it forward filtering density} for system \eqref{eq2}. This definition is motivated by the
following
\end{remark}
\begin{theorem}\label{th2}
Let $(Z^{t,z}_{s},Y_{s})_{s\in[t,T]}$ denote the solution of system \eqref{eq2} with initial
condition $Z^{t,z}_{t}=z$. {Under Assumptions \ref{A1bis}, \ref{A2} and \ref{A3}}, for any
$\phi\in {bC(\R^{2})}$ we have
\begin{equation}\label{ae7}
 E\left[\phi(Z_{T}^{t,z})\mid {\F^{Y}_{t,T}}\right]=
 \int_{\R^{2}}\hat{\mathbf{\pG}}(t,z;T,\z)\phi(\z)d\z, \qquad (t,z)\in [0,T]\times \R^2.
\end{equation}
\end{theorem}
\proof By Remark \ref{ar1}, $\int_{\R^{2}}\hat{\mathbf{\pG}}(t,z;T,\z)\phi(\z)d\z\in
m\F^{Y}_{t,T}$. We prove that, for any {bounded and $\F^{Y}_{t,T}$-measurable random variable
$G$,}
we have
\begin{equation}\label{ae5}
  E\left[G\phi(Z_{T}^{t,z})
  \right]=
 E\left[G (\r_T^{t,z})^{-1} \int_{\R^{2}}{\mathbf{\pG}}(t,z;T,\z)\phi(\z)d\z\right],
\end{equation}
with $\r^{t,z}$ as in \eqref{ae11}. From 
\eqref{ae5} with $\phi\equiv 1$ it will follow that
\begin{equation}
 E\left[(\r_T^{t,z})^{-1}\mid
 \F^{Y}_{t,T}\right]=\left(\int_{\R^{2}}\mathbf{\pG}(t,z;T,\z)d\z\right)^{-1}
\end{equation}
and therefore also \eqref{ae7} will follow from \eqref{ae5}.

By a standard approximation argument, it is enough to take $\phi$ in the class of test functions
and $G$ of the form $G=e^{-\int_t^Tc_{s}ds}$
where $c_{s}=c(s,Y_{s})$ with $c=c(s,y)$ 
being a {smooth, bounded and non-negative function} on $[t,T]\times\R$. Thus, we are left with the
proof of the following identity:
\begin{equation}\label{to_prove}
 E\left[e^{-\int_t^Tc_{s}ds}\phi(Z_{T}^{t,z})\right]=
 E\left[e^{-\int_t^Tc_{s}ds}(\r_T^{t,z})^{-1}\int_{\R^{2}}{\mathbf{\pG}}(t,z;T,\z)\phi(\z)d\z\right].
\end{equation}
To this end, we consider the deterministic backward Cauchy problem
\begin{equation}\label{Cauchy_pb}
 {f\left(s,e^{(T-s)B}\z,y\right)=\phi(\z)+\int_s^T(\fAA_\t-c(\t,y))f\left(\t,e^{(T-\t)B}\z,y\right)d\t, \qquad
 (s,\z,y)\in[t,T]\times \R^{2}\times\R,}
\end{equation}
{where
\begin{equation}\label{fLL}
 \fAA_{\t}:=\frac{1}{2}\left(|\sV(\t,\z,y)|^{2}\p_{\n\n}+2\sY(\t,y){\bar\sV(\t,\z,y)}\p_{\n y}+\sY^{2}(\t,y)\p_{yy}\right)
 +\bV(\t,\z,y)\p_{\n}+h(\t,\z,y)\p_{y}.
 \end{equation}
} In differential form, \eqref{Cauchy_pb} reads as
  $$
  \begin{cases}
    -d_{{\mathbf{B}}}f(s,\z,y)=\left(\fAA_sf(s,\z,y)-c(s,y)f(s,\z,y)\right)ds, \\
    f(T,\z,y)=\phi(\z).
  \end{cases}
  $$
{Corollary \ref{L_Cauchy}} ensures existence and estimates of a strong solution $f$ to
\eqref{Cauchy_pb}.

Next, we consider the process
 $$M_{s}^{t,z}:=e^{-\int_{t}^{s} c_{\t}d\t} (\r_s^{t,z})^{-1}
 \int_{\R^{2}}{\mathbf{\pG}}(t,z;s,\z)f(s,\z,Y_s)d\z,\qquad s\in [t,T],$$
{where $M_{t}^{t,z}$ is defined by continuity.} By definition, we have
 $$M_T^{t,z}=e^{-\int_t^Tc_{s}ds}(\r_T^{t,z})^{-1}
 \int_{\R^{2}}{\mathbf{\pG}}(t,z;T,\z)\phi(\z)d\z.$$
On the other hand, by the Feynman-Kac theorem 
we have
 $${M_{t}^{t,z}=f(t,z,Y_t) =E\left[e^{-\int_{t}^Tc_{s}ds}\phi(Z_T^{t,z})\mid Y_{t}\right].}$$
Hence to prove \eqref{to_prove} it suffices to check that $M=(M_s^{t,z})_{s\in[t,T]}$ is a
martingale: to this end, we prove the representation
\begin{align}\label{SviluppoM}
  M_T^{t,z}&={M_t^{t,z}+\int_{t}^{T}G^{t,z}_{s}{dW^{1}_s}},\\ \label{SviluppoM1}
  G^{t,z}_{s}&=e^{-\int_{t}^{s} c_{\t}d\t} (\r_s^{t,z})^{-1}
 \int_{\R^{2}}{\mathbf{\pG}}(t,z;s,\z)\left(\fG_s+\sY_s\p_y\right)f(s,\z,Y_s)d\z,\qquad
 s\in[t,T],
\end{align}
and conclude by showing that
\begin{equation}\label{StimaM}
   E\left[\int_{t}^{T}|G^{t,z}_{s}|^{2}ds\right]<\infty.
\end{equation}
We first compute the stochastic differential $d_{\mathbf{B}}f(s,\z,Y_s)$: {by Corollary
\ref{L_Cauchy}} we have
\begin{align}
 d_{\mathbf{B}}f(s,\z,Y_s)&=\left(-\fAA_s+\frac{1}{2}\sY_s^2\p_{yy}+c_s\right)f(s,\z,Y_s)ds+\p_yf(s,\z,Y_s)dY_s\\
 &=\left(-\fAA_s+\frac{1}{2}\sY_s^2\p_{yy}+{\bY_s}(Z_s)\p_y+c_s\right)f(s,\z,Y_s)ds+
 \sY_s \p_yf(s,\z,Y_s){dW^{1}_s}.
\end{align}
On the other hand, 
we have 
\begin{align}
d_{\mathbf{B}}\mathbf{\pG}(t,z;s,\z)&=\fA^{\ast}_s\mathbf{\pG}(t,z;s,\z)ds+
  \fG^{\ast}_s\mathbf{\pG}(t,z;s,\z)\frac{dY_s}{\sY_s}\\
 &=\left(\A^{\ast}_s+\tilde{h}_s(Z_s)\fG^{\ast}_s \right)\mathbf{\pG}(t,z;s,\z)ds+ \fG^{\ast}_s\mathbf{\pG}(t,z;s,\z){dW^{1}_s.}
\end{align}
Then, by It\^o formula we have 
\begin{align}
 d_{\mathbf{B}}\left(f(s,\z,Y_s)\mathbf{\pG}(t,z;s,\z)\right)
=I_{1}(t,z;s,\z)ds+I_{2}(t,z;s,\z){dW^{1}_{s}}
\end{align}
where
\begin{align}
 I_{1}(t,z;s,\z)&=f(s,\z,Y_s) \left(\A^{\ast}_s+\tilde{h}_s(Z_s)\fG^{\ast}_s \right)\mathbf{\pG}(t,z;s,\z)
 \\ &\quad+\mathbf{\pG}(t,z;s,\z)\left(-\fAA_s+\frac{1}{2}\sY_s^2\p_{yy}+\bY_s(Z_s)\p_y+c_s\right)f(s,\z,Y_s)
 + \sY_s \fG^{\ast}_s\mathbf{\pG}(t,z;s,\z)\p_yf(s,\z,Y_s),\\
 I_{2}(t,z;s,\z)&=f(s,\z,Y_s)\fG^{\ast}_s\mathbf{\pG}(t,z;s,\z)+ \sY_s \mathbf{\pG}(t,z;s,\z)
 \p_yf(s,\z,Y_s).
\end{align}
This means that for any $s\in(t,T]$ we have
\begin{align}
 f(T,{{\gY_{T-s}(\z)}},Y_T)\mathbf{\pG}(t,z;T,{{\gY_{T-s}(\z)}})&=
  f(s,\z,Y_s)\mathbf{\pG}(t,z;s,\z)\\
  &\quad+\int_{s}^{T}I_{1}(t,z;\t,{{\gY_{\t-s}(\z)}})d\t
  +\int_{s}^{T}I_{2}(t,z;\t,{{\gY_{\t-s}(\z)}}){dW^{1}_{\t}}.
\end{align}
Next, we integrate over $\R^2$ the previous identity and apply the standard and stochastic
Fubini's theorems (see, for instance, \cite{MR3839316}, Chapter 1) to get
\begin{align}
 \int_{\R^{2}}f(T,{{\gY_{T-s}(\z)}},Y_T)\mathbf{\pG}(t,z;T,{{\gY_{T-s}(\z)}})d\z&=
  \int_{\R^{2}}f(s,\z,Y_s)\mathbf{\pG}(t,z;s,\z)d\z\\
  &\quad+\int_{s}^{T}\int_{\R^{2}}I_{1}(t,z;\t,{{\gY_{\t-s}(\z)}})d\z d\t\\
  &\quad+\int_{s}^{T}\int_{\R^{2}}I_{2}(t,z;\t,{{\gY_{\t-s}(\z)}})d\z {dW^{1}_{\t}}. \label{equality}
\end{align}
By the upper bounds \eqref{t_e1}-\eqref{t_e3} of the fundamental solution, the estimates of the
solution $f$ and its derivatives in Corollary \ref{L_Cauchy}, the boundedness of the coefficients
and the non-degeneracy Assumption \ref{A1bis}, we have
\begin{align}
\int_s^T\int_{\R^2}|I_1(t,z;\t,\z)|d\z d\t \le \int_s^{T}\frac{C}{(T-\t)^{\frac
12}(s-t)}\int_{\R^2} \G_{\l}(\t-t,z;\t,\FIWi_{t,\t}(\z)-\gvar_\t^{t,z})d\z d\t \le C'
\frac{(T-s)^{\frac 12}}{s-t}, \intertext{and, analogously}
\int_s^T\left(\int_{\R^2}|I_2(t,z;\t,\z)|d\z \right)^2d\t \le
\int_s^{T}\left(\frac{C}{(s-t)^{\frac 12}}\int_{\R^2}
\G_{\l}(\t-t,z;\t,\FIWi_{t,\t}(\z)-\gvar_\t^{t,z})d\z\right)^2d\t \le \frac{C'}{s-t},
\end{align}
for some positive {random variables $C,C'$}. This justifies the use of Fubini's theorems.

Now, from equality \eqref{equality} we perform the change of variable $\z'=\gY_{\t-s}(\z)$, which
has Jacobian matrix $\text{Id}_{2\times 2}+(\t-s)B$; since its determinant is equal to one for any
$\t$, we get
\begin{align}
 \int_{\R^{2}}f(T,\z,Y_T)\mathbf{\pG}(t,z;T,\z)d\z&=
  \int_{\R^{2}}f(s,\z,Y_s)\mathbf{\pG}(t,z;s,\z)d\z\\
  &\quad+\int_{s}^{T}\int_{\R^{2}}I_{1}(t,z;\t,\z)\z d\t
  +\int_{s}^{T}\int_{\R^{2}}I_{2}(t,z;\t,\z)d\z {dW^{1}_{\t}}.
\end{align}
Integrating by parts and using the identity
%
%
\begin{align}
 &\int_{\R^2}\left(f(s,\z,Y_s)\fA^{\ast}_s\mathbf{\pG}(t,z;s,\z)+\mathbf{\pG}(t,z;s,\z)\frac{1}{2}\sY^2_s\p_{yy}f(s,\z,Y_s)+
 \sY_s\fG^{\ast}_s\mathbf{\pG}(t,z;s,\z)\p_yf(s,\z,Y_s)\right)d\z \\
 &=\int_{\R^2}\mathbf{\pG}(t,z;s,\z)\left(\fA_s+\frac{1}{2}\sY^2_s\p_{yy}+\sY_s\sVii_s\p_{y\v}+\bY_s(\z,Y_s)\p_y\right)f(s,\z,Y_s)d\z \\
 &=\int_{\R^2}\mathbf{\pG}(t,z;s,\z)\fAA_sf(s,\z,Y_s)d\z,
\end{align}
we get
\begin{align}
 \int_{\R^{2}}f(T,\z,Y_T)\mathbf{\pG}(t,z;T,\z)d\z&=
  \int_{\R^{2}}f(s,\z,Y_s)\mathbf{\pG}(t,z;s,\z)d\z\\
 &\quad+\int_{s}^{T}\int_{\R^2}\mathbf{\pG}(t,z;\t,\z)\left(\hY_\t(Z_\t)\fG_\t+\bY_\t(Z_\t)\p_y+c_\t\right)f(\t,\z,Y_\t)d\z
 d\t\\
 &\quad+ \int_{s}^{T}\int_{\R^2}\mathbf{\pG}(t,z;\t,\z)\left(\fG_\t+\sY_\t\p_y\right)f(\t,\z,Y_\t) d\z {dW^{1}_\t}.
\end{align}
Eventually, we multiply the expression above by $e^{-\int_{t}^{s} c_{\t}d\t}(\r_s^{t,z})^{-1}$:
since
\begin{align}
 d \left(e^{-\int_{t}^{s} c_{\t}d\t}(\r_s^{t,z})^{-1}\right)&= e^{-\int_{t}^{s} c_{\t}d\t}(\r_s^{t,z})^{-1}
 \left(-c_sds-\hY_s(Z_s){dW^{1}_s}\right),\\
 d\langle e^{-\int_{t}^{\cdot}
  c_{\t}d\t}(\r_{\cdot}^{t,z})^{-1},
 \int_{\R^2}f(\cdot,\z,Y_{\cdot})\mathbf{\pG}(t,z;\cdot,\z)d\z\rangle_s
 &= -\int_{\R^2}\mathbf{\pG}(t,z;s,\z)\left(\hY_s(Z_s)\fG_s+\bY_s(Z_s)\p_y\right)f(s,\z,Y_s)
 d\z ds,
\end{align}
by It\^o formula, for $s\in(t,T]$ we have
\begin{align}
 M_T^{t,z}&=e^{-\int_{t}^{T} c_{\t}d\t}(\r_T^{t,z})^{-1}\int_{\R^{2}}f(T,\z,Y_T)\mathbf{\pG}(t,z;T,\z)d\z \\
 &=
 M_s^{t,z}+ \int_{s}^{T}e^{-\int_{t}^{\t} c_{\r}d\r}(\r_\t^{t,z})^{-1} \int_{\R^2}
 \mathbf{\pG}(t,z;\t,\z)\left(\fG_\t+\sY_\t\p_y\right)f(\t,\z,Y_\t) d\z {dW^{1}_\t} \\
 &=M_s^{t,z}+ \int_{s}^{T} {G^{t,z}_{\t}}\label{ae10}
{dW^{1}_\t}.
\end{align}
with $G^{t,z}_{\t}$ as in \eqref{SviluppoM1}. Now, {again} by the estimates of the fundamental
solution (cf. Theorem \ref{TH1}), the estimates of the solution $f$ and its derivatives (cf.
Corollary \ref{L_Cauchy}), the boundedness of the coefficients and the non-degeneracy condition
\eqref{A1bis}, we deduce the estimate
\begin{align}\label{Uniform_integrability}
 |G^{t,z}_{\t}|&\le C(\r_\t^{t,z})^{-1}\int_{\R^2}\Gamma_\l(\t-t,g^{\text{\rm\tiny IW},-1}_{\t,t}(\z)-\gvar^{t,z}_\t)d\z\le {C'}
\end{align}
for some positive constants $C,C'$. This implies \eqref{StimaM} and concludes the proof.
\endproof

\subsection{Backward filtering SPDE}\label{bSPDE} 
%
%
%
As in the previous section, in order to apply the general results of Section 1 to the filtering
SPDE for system \eqref{eq2}, we impose the following conditions:
\begin{assumption}[Regularity]\label{A2ii} The coefficients of \eqref{eq2} are such that
${\bar\sV\in bC^{3+\a}_{0,T}(\R^3)}$, $\sVi\in bC^{\a}_{0,T}(\R^3)$, $\sY\in bC^{3+\a}_{0,T}(\R)$,
$b\in bC^{0}_{0,T}(\R^3)$, $h\in bC^{2}_{0,T}(\R^3)$.
\end{assumption}
\begin{assumption}[Flattening at infinity]\label{A3ii} There exist two positive constants
$\e,M$ such that
\begin{align}
  \sup_{t\in[0,T]}\left({\langle \bar\sV(t,\cdot,\cdot)\rangle_{\e,\b}+\langle \bar\sV(t,\cdot,\cdot)\rangle_{1/2+\e,\b'}}
  +\langle \sY(t,\cdot)\rangle_{\e,\b}+\langle \sY(t,\cdot)\rangle_{1/2+\e,\b'}
  +\langle h(t,\cdot,\cdot)\rangle_{1/2,\b}\right)&\le M
\end{align}
for $|\b|=1$ and $|\b'|=2,3$.
\end{assumption}

The backward filtering SPDE for system \eqref{eq2} reads
\begin{equation}\label{Backward_eq2}
 -d_{{\mathbf{B}}}u_t(z,y)=\fAA_{t}
 u_t(z,y)dt+\fGG_{t}u_t(z,y)\star\frac{dY_t}{\sY(t,y)}, \qquad
 \mathbf{B}:=\p_t+v\p_x,
\end{equation}
where $z=(x,v)$ and {\begin{align}\label{fLL2}
 \fAA_t&:=\frac{1}{2}\left(|\sV(t,z,y)|^{2}\p_{vv}+2\sY(t,y){\bar\sV(t,z,y)}\p_{v y}+\sY^{2}(t,y)\p_{yy}\right)
 +\bV(t,z,y)\p_v+h(t,z,y)\p_{y},\\
  \label{fGG}
 \fGG_{t}&:={\bar\sV(t,z,y)}\p_v+\sY(t,y)\p_{y}+\hY(t,z,y),\qquad
 \hY(t,z,y):=\frac{\bY(t,z,y)}{\sY(t,y)}.
\end{align}}
Before presenting the main result of this section, we comment on the existence of solutions to
\eqref{Backward_eq2}.
Let $(Z_{s}^{t,\ww},Y_{s}^{t,\ww},{\r}^{t,\ww,\y}_{s})_{s\in[t,T]}$ be the solution, starting at
time $t$ from $(\ww,\y)$, of the system of SDEs
\begin{equation}
{\begin{cases}\label{eq4}
 dZ_t=B Z_t dt +\mathbf{e}_2({b(t,Z_t,Y_t)dt}+\bar\sV(t,Z_t,Y_t)dW^1_t+\hat\sV(t,Z_t,Y_t)dW^2_t),\\ 
 dY_t=\bY(t,Z_t,Y_t)dt+\sY(t,Y_t)d{W}^{1}_t, \\
 d\r_t=\hY(t,Z_t,Y_t)^2\r_tdt+\hY(t,Z_t,Y_t)\r_td{W}^{1}_t.
\end{cases}}
\end{equation}
{By Girsanov's theorem, the process
\begin{align}
 \tilde{W}^{t,\ww}_{s}:\!\!&={\int_{t}^{s}\sY^{-1}(\t,Y^{t,\ww}_\t)dY^{t,\ww}_\t}\\
 &={W^{1}_{s}-W^{1}_t}+\int_t^{s}\hY(\t,Z^{t,\ww}_{\t},Y^{t,\ww}_{\t})d\t, \qquad s\in [t,T],
\end{align}
is a Brownian motion w.r.t the measure $Q^{t,\ww}$ defined by
$dQ^{t,\ww}=(\r_T^{t,\ww,1})^{-1}dP$.} Notice also that $(\tilde{W}^{t,\ww}_{s})_{s\in[t,T]}$ is
adapted to $(\F^{Y}_{t,s})_{s\in[t,T]}$ where {\bf $\F^{Y}_{t,s}=\s(Y^{t,z,y}_{\t},\, t\le \t\le
s)$}. Then equation \eqref{Backward_eq2} can be written in the equivalent form
\begin{equation}\label{Backward_eq3}
 -d_{{\mathbf{B}}}u_s(z,y)=\fAA_{s}
 u_s(z,y)ds+\fGG_{s}u_s(z,y)\star {d\tilde{W}_{s}^{t}}
\end{equation}
or, more explicitly,
\begin{equation}\label{spde_back4}
 u_{t}\left(\gY_{T-t}(z,y)\right)=u_{T}(z,y)+\int_{t}^{T}\fAA_{s}u_s(\gY_{T-s}(z,y))ds
 + \int_{t}^{T}\fGG_{s}u_s(\gY_{T-s}(z,y))\star d{\tilde{W}^{t}_{s}},\qquad t\in[0,T],
\end{equation}
where {$\gY_{s}(z,y)=\gY_{s}(x,v,y)=(x+sv,v,y)$}.
In \eqref{Backward_eq3} and \eqref{spde_back4}, we simply write $\tilde{W}^{t}_{s}$ instead of
$\tilde{W}^{t,\ww}_{s}$ because the starting point of the Brownian motion is irrelevant in the
stochastic integration. 
Theorem \ref{TH1} guarantees that a fundamental solution $\cev{\pG}=\cev{\pG}(t,z,y;s,\z,\y)$ for
\eqref{Backward_eq3} exists and satisfies estimates \eqref{t_e4}, \eqref{t_e5} and \eqref{t_e6}.
Moreover, $t\mapsto \cev{\pG}(t,z,y;T,\z,\y)$ is adapted to $({\F^{Y}_{t,T}})_{t\in[0,T]}$.
The main result of this section is the following
\begin{theorem}\label{at1} Let
$(Z^{t,z,y}_T,Y^{t,z,y}_T)$ denote the solution of system \eqref{eq2} starting from $(z,y)$ at
time $t\in[0,T)$ and $\phi\in bC(\R^{3})$. {Under Assumptions \ref{A1bis}, \ref{A2ii} and
\ref{A3ii}}, we have
\begin{equation}\label{ae2}
 E\left[\phi(Z_{T}^{t,z,y},Y_{T}^{t,z,y})\mid\F^{Y}_{t,T}\right]=\frac{u^{(\phi)}_{t}(z,y)}{u^{(1)}_{t}(z,y)},
 \qquad (t,z,y)\in [0,T]\times \R^2\times \R,
\end{equation}
where $u^{(\phi)}_{t}$ denotes the solution to \eqref{Backward_eq2} with final datum
$u^{(\phi)}_T=\phi$.
\end{theorem}
\begin{definition}[\bf Backward filtering density]
The normalized process 
\begin{equation}\label{ae1}
  \bar{\pG}(t,z,y;T,\z,\y)=\frac{\cev{\pG}(t,z,y;T,\z,\y)}{\int\limits_{\R^{3}}\cev{\pG}(t,z,y;T,\z_{1},\y_{1})d\z_{1}
  d\y_{1}},
\end{equation}
for $0\le t<T$ and $(z,y),(\z,\y)\in\R^{2}\times\R$, is called the backward filtering density of
system \eqref{eq2}. By Theorem 
\ref{at1}, we have
\begin{equation}\label{form1}
 E\left[\phi(Z_{T}^{t,z,y},Y_{T}^{t,z,y})\mid {\F^{Y}_{t,T}}\right]=
 \int\limits_{\R^{3}}\bar{\pG}(t,z,y;T,\z,\y)\phi(\z,\y)d\z d\y, \qquad (t,z,y)\in [0,T]\times
 \R^2\times\R,
\end{equation}
for any $\phi\in bC(\R^{3})$.
\end{definition}
\begin{remark}
Notice that formulas \eqref{ae2} and \eqref{form1} represent the conditional expectation in terms
of solutions to the Cauchy problem for the backward filtering SPDE. This is not the case for
formula \eqref{ae7} in the forward case.
\end{remark}


In the rest of the section we sketch the proof of Theorem \ref{at1}. First, notice that under
$Q^{t,\ww}$ we have
\begin{equation}\label{eq3}
 \r^{t,\ww,\y}_s=\y\exp\left(\int_t^s\hY(\t,Z^{t,\ww}_\t,Y^{t,\ww}_\t){d\tilde{W}^{t}_\t}-
 \frac{1}{2}\int_t^s\hY(\t,Z^{t,\ww}_\t,Y^{t,\ww}_\t)^2d\t\right), \qquad s\in [t,T],
\end{equation}
and system \eqref{eq4} reads
\begin{equation}\label{ae3}
\begin{cases}
dZ^{t,\ww}_s=\fB(s,Z^{t,\ww}_s,Y^{t,\ww}_s)ds +
\mathbf{e}_2\left({\sVi(s,Z^{t,\ww}_s,Y^{t,\ww}_s)dW^2_t+\bar\sV(s,Z^{t,\ww}_s,Y^{t,\ww}_s){d\tilde{W}^{t}_s}}\right),\\
dY^{t,\ww}_s=\sY(s,Y^{t,\ww}_s){d\tilde{W}^{t}_s},\\
d\r^{t,\ww,\y}_s=\hY(s,Z^{t,\ww}_s,Y^{t,\ww}_s)\r^{t,\ww,\y}_s{d\tilde{W}^{t}_s},
\end{cases}
\end{equation}
where $\fB(s,z,y)=Bz+\mathbf{e}_2({\bV(s,z,y)}-\hY(s,z,y){\bar\sV(s,z,y)})$. Recalling the
notation $z=(x,v)\in\R^{2}$ and omitting the arguments of the coefficients for brevity, the
correspondent characteristic operator is
\begin{align}
 \fLLL=\frac{1}{2}\left({|\sV|^{2}}\p_{vv}+\sY^2\p_{yy}+\y^2\hY^2\p_{\y\y}+
 {2\bar{\sV}\sY\p_{vy}+2\y\bar{\sV}\hY\p_{v\y}}+2\y\sY\hY\p_{y\y}\right) +\langle
 \fB,\nabla_{z}\rangle.
\end{align}
We write the backward diffusion SPDE for system \eqref{ae3}. {Assuming that $\phi$ is smooth} and
letting $V_s(\ww):=\phi(Z_{T}^{s,\ww},Y_{T}^{s,\ww})$, by Corollary \ref{c1} we have
\begin{align}
 -d(V_s(\ww)\r_{T}^{s,\ww,\y})&=\fLLL(V_s(\ww)\r_{T}^{s,\ww,\y})ds+\p_v(V_s(\ww)\r_{T}^{s,\ww,\y})
\left({\sVi(s,\ww)\star dW_s^2+\bar\sV(s,\ww)\star {d\tilde{W}^{t}_s}}\right)\\
 &\quad +\p_y(V_s(\ww)\r_{T}^{s,\ww,\y})\sY(s,\ww)\star {d\tilde{W}^{t}_s}+
 \p_\y(V_s(\ww)\r_{T}^{s,\ww,\y})\y\hY(s,\ww)\star {d\tilde{W}^{t}_s}
\intertext{(noting that $\p_{\y}Z^{t,\ww}_{T}=\p_{\y}Y^{t,\ww}_{T}=\p_{\y\y}{\r}^{t,\ww,\y}_{T}=0$
and $\y\p_{\y}{\r}^{t,\ww,\y}_{T}={\r}^{t,\ww,\y}_{T}$)}
  &=\frac{1}{2}\left({{|\sV(t,\ww)|^{2}}}\p_{vv}+\sY^2(t,\ww)\p_{yy}+{2\bar\sV(t,\ww)\sY(t,y)\p_{vy}}\right)(V_s(\ww)\r_{T}^{s,\ww,\y})ds\\
  &\quad +\Big(\hY(s,\ww)({\bar\sV(t,\ww)\p_v+\sY(t,y)\p_y)}+\langle \fB(t,z,y),\nabla_{z}\rangle\Big)(V_s(\ww)\r_{T}^{s,\ww,\y})ds\\
 &\quad +{\sVi(s,\ww) \p_v(V_s(\ww)\r_{T}^{s,\ww,\y})\star {dW_s^2}}\\
 &\quad +\left({\bar{\sV}(s,\ww)}\p_v+\sY(s,\ww)\p_y+\hY(s,\ww)\right)(V_s(\ww)\r_{T}^{s,\ww,\y})\star {d\tilde{W}^{t}_s}
\intertext{(noting that $\hY(s,\ww)({\bar{\sV}(t,\ww)\p_v}+\sY(t,y){\p_y})+\langle
\fB(t,z,y),\nabla_{z}\rangle={v\p_x+\bV(t,z,y)\p_v+\bY(t,z,y)\p_y}$)}
 &={\fLL(V_s(\ww)\r_{T}^{s,\ww,\y})ds+\sVi(s,\ww) \p_v(V_s(\ww)\r_{T}^{s,\ww,\y})\star
 {dW_s^2}}\\
 &\quad +\left({\bar\sV(s,\ww)}\p_v+\sY(s,y){\p_y}+\hY(s,\ww)\right)(V_s(\ww)\r_{T}^{s,\ww,\y})\star
 {d\tilde{W}^{t}_s}.
\end{align}
where ${\fLL}=\fAA_t+v\p_x$, with $\fAA_t$ as in \eqref{fLL2}, is the infinitesimal generator of
$(Z_t,Y_t)$. Therefore we have
\begin{align}
\phi(Z_{T}^{t,\ww},Y_{T}^{t,\ww})\r_{T}^{t,\ww,1}-\phi(\ww)&=
V_t(\ww)\r_{T}^{t,\ww,1}-V_{T}(\ww)\r_{T}^{{T},\ww,1}\\ &=
\int_t^{T}\fLL(V_s(\ww)\r_{T}^{s,\ww,\y})ds+\int_t^{T}{\sVi(s,\ww)\p_v(V_s(\ww)\r_{T}^{s,\ww,\y})\star
{d{W}_s^2}}\\ &\quad + \int_t^{T} \fGG_s
(V_s(\ww)\r_{T}^{s,\ww,\y})
\star {d\tilde{W}^{t}_s}.\label{ae4}
\end{align}
Now we take the conditional expectation in \eqref{ae4} and exploit the fact that {$W^{2}$ is
independent of $\F^{Y}_{t,T}$ under $Q^{t,\ww}$} (this follows from the crucial assumption that
$\sY$ is a function of $t,y$ only): setting
 \begin{equation}
 u^{(\phi)}_{t}(z,y)=E^{Q^{t,z,y}}\left[V_{t}(z,y){\r}^{t,z,y,1}_{T}\mid\F^{Y}_{t,T}\right],
\end{equation}
and applying the standard and stochastic Fubini's theorems, we directly get the filtering equation
\begin{align}
 u^{(\phi)}_{t}(z,y)&=\phi(z,y)+ \int_{t}^{T}\fLL_s u^{(\phi)}_{s}(z,y)ds +\int_{t}^{T}
 \fGG_s u^{(\phi)}_{s}(z,y)
 \star
 \frac{dY_s^{t,z,y}}{\sY(s,y)}
\end{align}
which is equivalent to \eqref{Backward_eq2}. 
Analogously,
  $$u^{(1)}_{t}(z,y):=E^{Q^{t,z,y}}\left[{\r}^{t,z,y,1}_{T}\mid\F^{Y}_{t,T}\right]$$
solves the same SPDE with terminal condition $u^{(1)}_{T}(z,y)\equiv 1$. To conclude, it suffices
recall the Bayes representation for conditional expectations or the Kallianpur-Striebel's formula
(cf. \cite{MR3839316}, Lemma 6.1) according to which we have
\begin{equation}
E\left[\phi(Z_{T}^{t,z,y},Y_{T}^{t,z,y})\mid\F^{Y}_{t,T}\right]=
\frac{E^{Q^{t,z,y}}\left[\phi(Z_{T}^{t,z,y},Y_{T}^{t,z,y}){\r}^{t,z,y,1}_{T}\mid\F^{Y}_{t,T}\right]}{E^{Q^{t,z,y}}
\left[{\r}^{t,z,y,1}_{T}\mid\F^{Y}_{t,T}\right]}.
\end{equation}
\endproof


\section{Proof of Theorem \ref{TH1}}\label{proofK}
As in \cite{pasc:pesc:19} the main ingredient in the proof of Theorem \ref{TH1} is the
It\^o-Wentzell formula that transforms the original SPDE into a PDE with random coefficients. In
this section we explain how to tweak the change of variables introduced in \cite{pasc:pesc:19} to
deal with the additional term $h$ and we also consider the backward equation. We set $d=1$ for
simplicity.

\subsection{It\^o-Wentzell change of variables}
We first recall some global estimates, proved in \cite{pasc:pesc:19}, Section 4, for $\gIW$,
$\bIW$ in \eqref{sde_forw}-\eqref{sde_back} and their derivatives under Assumptions \ref{ass1},
\ref{ass2} and \ref{ass3}.
\begin{lemma}\label{lemmaIW} For any $\bar{\a}\in [0,\a)$, we have ${\gIW_{t,\cdot}}\in{\mathbf{C}^{3+\bar{\a}}_{t,T}}$. Moreover there exists
$\e\in\left(0,\frac{1}{2}\right)$ and a random, finite constant $\mathbf{c}$ such that, with
probability one,
\begin{align}
 |\gIW_{t,s}(\x,\n)|&\le \mathbf{c} \sqrt{1+\x^2+\n^2},\\
 e^{-\mathbf{c}(s-t)^{\e}}\le \p_\v\gIW_{t,s}(\x,\n)&\le e^{\mathbf{c}(s-t)^{\e}},\\
 |\p_\x \gIW_{t,s}(\x,\n)|&\le \mathbf{c}(s-t)^{\e},\\
 |\p^{\b}\gIW_{t,s}(\x,\n)|&\le \frac{\mathbf{c}(s-t)^{\e}}{\sqrt{1+\x^2+\n^2}},
\end{align}
for any $(\x,\n)\in\R^{2}$, $0\le t\le s\le T$ and $|\b|=2$. Analogous estimates hold for
$\bIW_{\cdot,s}\in\mathbf{C}^{3,\bar{\a}}_{0,s}$.
\end{lemma}

We introduce the ``hat'' operator which transforms any function $f_{s}(\x,\n)$, {$s\in[t,T]$,} 
into
\begin{equation}
\hat{f}_{t,s}(\x,\n):=f_{s}(\x,\gIW_{t,s}(\x,\n)).
\end{equation}
Let $u_{s}(\x,\n)$ a solution to \eqref{spde_forw} {on $[t,T]$}. Then we define
\begin{equation}
v_{t,s}(\z):=\hat{\r}_{t,s}(\z)\hat{u}_{t,s}(\z), \qquad {{\hat\r_{t,s}(\z):=\exp\left(-\int_{t}^s
\hat{h}_{\t}(\z)dW_\t-\frac{1}{2}\int_{t}^s\hat{h}_{\t}^2(\z)d\t \right)}}.
\end{equation}
We have the following
\begin{proposition} \label{Prop_IW}
$u_s$ is a solution to the SPDE \eqref{spde_forw} {on $[t,T]$} if and only if $v_{t,s}$ is a
solution on $[t,T]$ to the PDE with random coefficients
\begin{equation}\label{PDE_FORW}
 d_{\mathbf{\hat{B}}}v_{t,s}(\z)=\left(a^{\ast}_{t,s}(\z)\p_{\n\n}v_{t,s}+
 b_{t,s}^{\ast}(\z)\p_{\n}v_{t,s}(\z)+{c}^{\ast}_{t,s}(\z)v_{t,s}(\z)\right)ds,
 \qquad \mathbf{\hat{B}}=\p_s+\Qg_{t,s},
\end{equation}
where 
\begin{equation}\label{aee13}
  \Qg_{t,s}=\Qg_{t,s}(\x,\n):=(\gIW_{t,s})_1(\x,\n)\p_{\x}-(\gIW_{t,s}(\x,\n))_1(\p_{\n}\gIW_{t,s})^{-1}(\x,\n){\p_{\x}
  \gIW_{t,s}(\x,\n)}\p_{\n},
\end{equation}
is the first order operator identified with the vector field in \eqref{flusso2} (with $d=1$) and
the coefficients $a^{\ast}_{t,\cdot}$, $b^{\ast}_{t,\cdot}$, $c^{\ast}_{t,\cdot}$ are defined in
\eqref{coeff_deterministic} below. Moreover, $a^{\ast}_{t,\cdot}\in \mathbf{bC}^{\a}_{t,T}$,
$b^{\ast}_{t,\cdot}, c^{\ast}_{t,\cdot}\in \mathbf{bC}^{0}_{t,T}$, {$\Qg_{t,\cdot}\in
{\mathbf{C}^{0,1}_{t,T}}$}, $\p_\n(\Qg_{t,\cdot})_1\in \mathbf{bC}^{\bar{\a}}_{t,T}$ {for any
$\bar{\a}\in [0,\a)$,} and there exist two random, finite and positive constants $\mathbf{m}_1$,
$\mathbf{m}_2$ such that, for $s\in [t,T]$, $\z \in \R^2$, we have
\begin{align}\label{Ass5}
& \mathbf{m}_1^{-1}\le a^{\ast}_{t,s}(\z)\le \mathbf{m}_1, \qquad \mathbf{m}_2^{-1}\le \p_\n
(\Qg_{t,s}(\z))_1\le \mathbf{m}_2,
\end{align}
with probability one.
\end{proposition}
\proof 
By a standard regularization argument, we may assume $u\in \mathbf{C}_{t,T}^2$ so that equation
\eqref{spde_forw} can be written in the usual It\^o sense, namely
$$du_{s}(\z)=(\A_{s,\z}-\n_1\p_\x)u_{s}(\z)ds+\mathcal{G}_{s,\z}u_{s}(s)dW_s. $$ By the standard
It\^o-Wentzell formula (see for instance \cite{MR3839316}, Theorem 1.17), and the chain rule we
have
\begin{align}\label{p_e1}
d\hat{u}_{t,s}&=\left(\widehat{\A_{s,\z}u}_{t,s}-\gIW_{t,s}\widehat{\p_1u}_{t,s}+\frac{1}{2}\hat{\s}^2_{t,s}\widehat{\p_2^2
u}_{t,s}-\widehat{\p_2\mathcal{G}}_{s,\z}\hat{\s}_{t,s}\right)ds+\hat{h}_{t,s}\hat{u}_{t,s}dW_s\\
 &=\left(\mathfrak{L}_{t,s}-
 \Qg_{t,s}\right)\hat{u}_{t,s}dt+\hat{h}_{t,s}\hat{u}_{t,s}dW_s,
\end{align} 
where ${\mathfrak{L}_{t,s}}:= \au_{t,s}\p_{vv}+\bu_{t,s}\p_v+\bar{c}_{t,s}$ with
\begin{equation}\label{aee14}
 \begin{split}
 \au_{t,s}&=\frac{1}{2}(\p_\n\gIW_{t,s})^{-2}(\hat{a}_{t,s}-\hat{\s}^2_{t,s}),\\
 \bu_{t,s}&=(\p_\n\gIW_{t,s})^{-1}\left(\hat{b}_{t,s}-\hat{\s}_{t,s}\hat{h}_{t,s}-(\p_\n\gIW_{t,s})^{-1}{\hat{\s}_{t,s}\p_{\n}\hat{\s}_{t,s}}-
 \au_{t,s}\p_{\n\n}\gIW_{t,s}\right), \\
 \bar{c}_{t,s}&=\hat{c}_{t,s}-(\p_\n\gIW_{t,s})^{-1}{\hat{\s}_{t,s}\p_{\n}\hat{h}_{t,s}}.
 \end{split}
\end{equation}
Notice that the change of variable is well defined by the estimates of Lemma \ref{lemmaIW}. Next
we compute the product $v_{t,s}=\hat{\r}_{t,s}\hat{u}_{t,s}$: by the It\^o formula
$d\hat{\r}_{t,s}=-\hat{\r}_{t,s}\hat{h}_{t,s}dW_s$ and therefore 
\begin{align}\label{p_e2}
 dv_{t,s}(\z)&=\hat{\r}_{t,s}(\z)d\hat{u}_{t,s}(\z)+\hat{u}_{t,s}(\z)d\hat{\r}_{t,s}+ d\langle
 \hat{u}_{t,\cdot}(\z)\hat{\r}_{t,\cdot}(\z) \rangle_s \\
 &=\left(\hat{\r}_{t,s}(\z)\mathfrak{L}_{t,s}(\hat{\r}^{-1}_{t,s}v_{t,s})(\z)-
 \hat{\r}_{t,s}(\z)
 (\Qg_{t,s}(\hat{\r}^{-1}_{t,s}v_{t,s}))(\z)-\bar{h}_{t,s}^2(\z)v_{t,s}(\z)\right)ds.
\end{align}
Now we notice that
\begin{align}
 \hat{\r}_{t,s}(\z)(\Qg_{t,s}(\hat{\r}^{-1}_{t,s}v_{t,s}))(\z)&=
 (\Qg_{t,s}v_{t,s})(\z) + (\Qg_{t,s}\ln \hat{\r}^{-1}_{t,s})(\z)v_{t,s}(\z),
\end{align}
and eventually, by a standard application of the Leibniz rule, we get
 $$dv_{t,s}(\z)=\left(a^{\ast}_{t,s}(\z)\p_{vv}v_{t,s}(\z)-(\Qg_{t,s}
 v_{t,s})(\z)+ b_{t,s}^{\ast}(\z)\p_vv_{t,s}(\z)+c^{\ast}_{t,s}(\z)v_{t,s}(\z)\right)ds, $$
where
\begin{align}\label{coeff_deterministic}
 a_{t,s}^{\ast}&=\au_{t,s}=\frac{1}{2}(\p_\n\gIW_{t,s})^{-2}(\hat{a}_{t,s}-\hat{\s}^2_{t,s}),\\
 b_{t,s}^{\ast}&=\bu_{t,s}+2 \au_{t,s}\p_\n\ln \hat{\r}^{-1}_{t,s}, \\
 c^{\ast}_{t,s}&=\bar{c}_{t,s}+\bu_{t,s}\p_\n\ln \hat{\r}^{-1}_{t,s}+\au_{t,s}\left(\p_\n\ln
 \hat{\r}^{-1}_{t,s}+ \p^2_\n\ln \hat{\r}^{-1}_{t,s}\right)+\Qg_{t,s}\ln
 \hat{\r}^{-1}_{t,s}-{\hat{h}_{t,s}^2}.
\end{align}
The regularity of the coefficients and \eqref{Ass5} follow directly from
\eqref{coeff_deterministic}, Assumption \ref{ass2} and Lemma \ref{lemmaIW}.
\endproof
\begin{remark}\label{Hormander}
When the coefficients are smooth, condition \eqref{Ass5} ensures the validity of the weak
H\"ormander condition: indeed the vector fields $\sqrt{a^{\ast}_{t,\cdot}}\p_\n$ and
$\Qg_{t,\cdot}$, together with their commutator, span $\R^{2}$ at any point. In this case a smooth
fundamental solution to \eqref{PDE_FORW} exists by H\"ormander's theorem.
\end{remark}


In the backward case the computations are completely analogous since it only suffices to reverse
the time in equations \eqref{p_e1} and \eqref{p_e2}. Precisely, we introduce the ``check''
transform
\begin{equation}
 \check{f}_{t,s}(x,v):=f_{t}(\x,\bIW_{t,s}(x,v)),\qquad t\in [0,s],
\end{equation}
with $\bIW_{t,s}$ as in \eqref{sde_back}. For a solution $u_{t}=u_{t}(z)$ to \eqref{spde_back} on
$[0,s]$, we define
\begin{equation}
 v_{t,s}(z):=\check{\r}_{t,s}(z)\check{u}_{t,s}(z), \qquad {{\check\r_{t,s}(z):=\exp\left(-\int_{t}^s
\check{h}_{\t}(z)\star dW_\t-\frac{1}{2}\int_{t}^s\check{h}_{\t}^2(z)d\t \right)}},
\end{equation}
which solves, on $[0,s]$, the deterministic equation with random coefficients
\begin{equation}\label{PDE_BACK}
 -d_{\small{\cev{\mathbf{B}}}}v_{t,s}(z)=\left(\cev{a}^{\ast}_{t,s}(z)\p_{vv}v_{t,s}+
 \cev{b}_{t,s}^{\ast}(z)\p_{v}v_{t,s}(z)+\cev{c}^{\ast}_{t,s}(z)v_{t,s}(z)\right)dt,
 \qquad {\cev{\mathbf{B}}}=\p_t+\cev{\Qg}_{t,s},
\end{equation}
where $\cev{\Qg}_{t,s}$ and the coefficients are defined similarly to \eqref{aee13} and
\eqref{aee14}, exchanging the {\it hat}-  and  {\it check}-transforms  in the definitions. As for
the forward case, by Assumption \ref{ass2} and Lemma \ref{lemmaIW}, 
$a^{\ast}_{t,\cdot}\in \mathbf{b}\cev{\mathbf{C}}^{\a}_{t,T}$, $b^{\ast}_{t,\cdot},
c^{\ast}_{t,\cdot}\in \mathbf{b}\cev{\mathbf{C}}^{0}_{t,T}$,
{$\cev{\Qg}_{t,\cdot}\in\cev{\mathbf{C}}^{0,1}_{t,T}$}, $\p_v(\cev{\Qg}_{t,\cdot})_1\in
\mathbf{b}\cev{\mathbf{C}}^{\bar{\a}}_{t,T}$, {for any $\bar{\a}\in [0,\a)$}, and there exist two
random, finite and positive constant $\mathbf{m}_1$, $\mathbf{m}_2$ such that, for $t\in [0,s]$
and $z \in \R^2$, we have
\begin{align}
& \mathbf{m}_1^{-1}\le \cev{a}^{\ast}_{t,s}(z)\le \mathbf{m}_1, \qquad \mathbf{m}_2^{-1}\le \p_v
(\cev{\Qg}_{t,s}(z))_1\le \mathbf{m}_2,
\end{align}
with probability one, which ensures the weak H\"ormander condition to hold.

\subsection{The parametrix expansion}
Equations of the form \eqref{PDE_FORW} have been studied in \cite{pasc:pesc:19} by means of a
time-dependent parametrix expansion which takes into account the unbounded drift $\Qg$. The only
minor difference here is the presence of a term of order zero $c^{\ast}$ which, as we shall see,
does not modify the analysis substantially.

In this section we briefly resume the parametrix construction and show how it works in the
backward framework. For the sake of readability, here we reset the notations and rewrite equation
\eqref{PDE_FORW} as
\begin{equation}\label{aee16}
  \mathcal{A}_{s}v_{s}(\z)-\Qg_{s}v_{s}(\z)-\p_{s}v_{s}(\z)=0,\qquad s\in(t,T],\ \z\in\R^{2},
\end{equation}
where $\mathcal{A}_s$ is a second order operator of the form
  $$\mathcal{A}_s=a_s\p_{\n\n} + b_s\p_\n +c_s
  $$
and
$\Qg_{s}=
(\Qg_{s})_{1}\p_{\x}+(\Qg_{s})_{2}\p_{\n}$ is the vector field in \eqref{aee13}. For fixed
$(t_{0},z_{0})\in [t,T)\times \R^2$, we linearize $\Qg_{s}$ by setting
\begin{equation}\label{aee17}
  \Qg_{s}^{t_{0},z_{0}}(\z)=\Qg_{s}(\gvar_{s}^{t_{0},z_{0}})+\left(D\Qg_{s}\right)(\gvar_{s}^{t_{0},z_{0}})
  \left(\z-\gvar_{s}^{t_{0},z_{0}}\right)
\end{equation}
where
   {$$\gvar_{s}^{t_{0},z_{0}}=z_{0}+\int_{t_{0}}^{s}\Qg_{\t}(\gvar_{\t}^{t_{0},z_{0}})d\t,\qquad s\in[t_{0},T],$$}
and $D\Qg_s$ is the reduced Jacobian defined as
 $$D\Qg_s:=\begin{pmatrix}0 & \p_v(\Qg_{s})_1 \\ 0 & 0 \end{pmatrix}.$$
Then we consider the linearized version of \eqref{aee16}, that is
\begin{equation}\label{aee16lin}
 \mathcal{A}^{t_{0},z_{0}}_{s}v_{s}(\z)-\Qg_{s}^{t_{0},z_{0}}v_{s}(\z)-\p_{s}v_{s}(\z)=0,\qquad s\in(t,T],\ \z\in\R^{2},
\end{equation}
where
\begin{equation}\label{linearized_PDE_forw}
 \mathcal{A}^{t_{0},z_{0}}_s:=a_s(\gvar_s^{t_{0},z_{0}})\p_{\n\n}. 
\end{equation}
It turns out (cf. \cite{pasc:pesc:19}, Section 5) that, for any choice of the parameters
$(t_{0},z_{0})$, equation \eqref{linearized_PDE_forw} has a fundamental solution
$\pG^{t_{0},z_{0}}=\pG^{t_{0},z_{0}}(t,z;s,\z)$. Moreover, $\pG^{t_{0},z_{0}}$ has an explicit
Gaussian expression and satisfies the following estimates {for some $\l,\m>0$ that depend only on
the general constants of Assumptions \ref{ass1}, \ref{ass2} and \ref{ass3}}:
\begin{equation}\label{aee18}
\begin{split}
 \Gamma_{\l^{-1}}\left(s-t,\z-{\g}_{t,s}^{t_{0},z_{0}}(z)\right)\le\,
 \m\pG^{t_{0},z_{0}}(t,z;s,\z)&\le \Gamma_{\l}\left(s-t,\z-{\g}_{t,s}^{t_{0},z_{0}}(z)\right), \\
 |\p_\n \pG^{t_{0},z_{0}}(t,z;s,\x,\n)|&\le  \frac{\m}{\sqrt{s-t}}\Gamma_{\l}\left(s-t,(\x,\n)-{\g}_{t,s}^{t_{0},z_{0}}(z)\right),\\
 |\p_{\n\n}\pG^{t_{0},z_{0}}(t,z;s,\x,\n)|&\le  \frac{\m}{s-t}\Gamma_{\l}\left(s-t,(\x,\n)-{\g}_{t,s}^{t_{0},z_{0}}(z)\right),
\end{split}
\end{equation}
for $0\le t<s\le T$ and $z,\z\in\R^{2}$, with $\Gamma_{\l}$ as in \eqref{gammal} and
$s\mapsto\g_{t,s}^{t_{0},z_{0}}(z)$ defined by
 $${\g}_{t,s}^{t_{0},z_{0}}(z)=z+\int_t^{s}\Qg_{\t}^{t_{0},z_{0}}({\g}_{\t,t}^{t_{0},z_{0}}(z))d\t,\qquad s\in[t,T].$$

We introduce the so-called \textit{forward parametrix}
  $$Z(t,z;s,\z):=\pG^{t,z}(t,z;s,\z),\qquad 0\le t<s\le T,\ z,\z\in\R^{2},$$
that will be used as a first approximation of a fundamental solution $\pG$ of \eqref{aee16}.
Owing to the fact that {$\g^{t,z}_{t,s}(z)=\gvar^{t,z}_s$}, $Z$ satisfies estimates \eqref{aee18}. 

Now we set
\begin{align}
 H(t,z;s,\z)&:=\left(\mathcal{A}_s-\Qg_{s}-(\mathcal{A}^{{t,z}}_s-\Qg_{s}^{{t,z}})\right)Z(t,z;s,\z)
\end{align}
and notice that
\begin{align}
 |H (t,z;s,\z)|&\le |a_s(\z)-a_s(\g_s^{t,z})||\p_{\n\n}Z(t,z;s,\z)|\\ &\quad+
 |(\Qg_s-\Qg_s^{t,z})Z(t,z;s,\z)| +
 |b_s(\z)||\p_{\n}Z(t,z;s,\z)|+|c_s(\z)||Z(t,z;s,\z)| \intertext{(since $\p_\n(\Qg_s)_1$ is
 $\bar{\a}$-H\"older continuous by Proposition \ref{Prop_IW})} &\le\m
 \left(\frac{|\z-\g_s^{t,z}|^{\a}}{s-t}+\frac{|\z-\g_s^{t,z}|^{1+\bar{\a}}}{(s-t)^{3/2}}
 +\frac{|\z-\g_s^{t,z}|^{\a}}{(s-t)^{1/2}}+1\right)\Gamma_{\l}(s-t,\z-{\gvar}_s^{t,z})
\intertext{(for some $\bar{\l}>\l$ and $\bar{\m}>\m$)}
 &\le
 \frac{\bar{\m}}{(s-t)^{1-\bar{\a}/2}}\Gamma_{\bar{\l}}(s-t,\z-{\gvar}_s^{t,z}).
\end{align}
Next, we set
\begin{align}
  \pG\otimes H (t,z;s,\z)&:=\int_t^s\int_{\R^2}H(t,z;\t,w)\pG(\t,w;s,\z)dw d\t.
\end{align}
A recursive application of the Duhamel principle shows that
\begin{align}\label{Expansion1}
 \pG(t,z;s,\z)&=Z(t,z;s,\z)+\pG\otimes H (t,z;s,\z)\\ &= Z(t,z;s,\z) + \sum_{k=1}^{N-1}Z\otimes
 H^{\otimes k} (t,z;s,\z)+\pG\otimes H^{\otimes N} (t,z;s,\z),\qquad N\ge1.
\end{align}
As $N$ tends to infinity we formally obtain a representation of $\pG$ as a series of convolution
kernels. Unfortunately, as already noticed in \cite{MR2659772} and \cite{pasc:pesc:19}, the
presence of the transport term makes it hard to control the iterated kernels uniformly in $N$,
as opposed to the classical parametrix method for uniformly parabolic PDEs. 
Thus the remainder $\pG\otimes H^{\otimes N}$ must be handled with a different technique, borrowed
from stochastic control theory: the rest of the proof proceeds exactly in the same way as in
\cite{pasc:pesc:19} to which we refer for a detailed explanation.

\medskip

Next, we consider the backward equation
\begin{equation}\label{aee16b}
  \cev{\mathcal{A}}_{t}u_{t}(z)+\cev{\Qg}_{t}u_{t}(z)+\p_{t}u_{t}(z)=0,\qquad t\in[0,s),\ z=(x,v)\in\R^{2},
\end{equation}
where $\cev{\mathcal{A}}_t$ is a second order operator of the form
  $$\cev{\mathcal{A}}_t=\cev{a}_t\p_{vv} + \cev{b}_t\p_{v} +\cev{c}_t,\qquad z=(x,v)\in\R^{2},
  $$
and
$\cev{\Qg}_{t}
=(\cev{\Qg}_{t})_{1}\p_{x}+(\cev{\Qg}_{t})_{2}\p_{v}.$
For a fixed $(s_{0},\z_{0})\in (0,s]\times\R^{2}$, we define the linearized version of
\eqref{aee16b}, that is
\begin{equation}\label{aee16lin_b}
 \cev{\mathcal{A}}^{s_{0},\z_{0}}_{t}u_{t}(z)+\cev{\Qg}_{t}^{s_{0},\z_{0}}u_{t}(z)+\p_{t}u_{t}(z)=0,\qquad t\in[0,s),\ z\in\R^{2},
\end{equation}
where the definition of $\cev{\Qg}_{t}^{s_{0},\z_{0}}$ is analogous to that of
$\Qg_{s}^{t_{0},z_{0}}$ in \eqref{aee17} and
\begin{equation}\label{linearized_PDE_back}
 \cev{\mathcal{A}}^{s_{0},\z_{0}}_t:=\cev{a}_t(\cev{\g}_t^{s_{0},\z_{0}})\p_{vv},\qquad
 \cev{\g}_{t}^{s_{0},\z_{0}}=\z_{0}+\int_{t}^{s_{0}}\cev{\Qg}_{\t}(\cev{\g}_{\t}^{s_{0},\z_{0}})d\t,\qquad t\in[0,s_{0}]. 
\end{equation}
Equation \eqref{aee16lin_b} has an explicit fundamental solution
$\cev{\pG}^{s_{0},\z_{0}}=\cev{\pG}^{s_{0},\z_{0}}(t,z;s,\z)$ of Gaussian type, that satisfies
estimates analogous to \eqref{aee18}.
The {\it backward parametrix} for {\eqref{aee16b}} is defined as
 $$\cev{Z}(t,z;s,\z)=\cev{\pG}^{s,\z}(t,z;s,\z),\qquad 0\le t<s\le T,\ z,\z\in\R^{2}.$$
As in the forward case, Duhamel principle yields the expansion
\begin{align}\label{Expansion2}
 \cev{\pG}(t,z;s,\z) = \cev{Z}(t,z;s,\z) + \sum_{k=1}^{N-1}\cev{Z}\otimes \cev{H}^{\otimes k}
 (t,z;s,\z)+\cev{\pG}\otimes \cev{H}^{\otimes N} (t,z;s,\z),\qquad N\ge 1,
\end{align}
where
$\cev{H}(t,z;s,\z)=\left(\cev{\mathcal{A}}_t+\cev{\mathbf{Y}}_{t}-\cev{\mathcal{A}}^{s_{0},\z_{0}}_t-
\cev{\mathbf{Y}}^{s_{0},\z_{0}}_{t}\right)\cev{Z}(t,z;s,\z)$ and the rest of the proof proceeds as
in the forward case. In particular, existence and estimates for the fundamental solutions  of
\eqref{PDE_FORW} and \eqref{PDE_BACK} (in the sense of Definitions \ref{d1} and \ref{d2}) follow
from the parametrix expansions \eqref{Expansion1} and \eqref{Expansion2}. Eventually, it suffices
to go back to the original variables to conclude the proof: we refer to \cite{pasc:pesc:19},
Section 6, for full details.

\subsection{Proof of Corollary \ref{L_Cauchy}}
{By Theorem \ref{TH1} there exists a fundamental solution $\cev{\pG}$ of equation
\eqref{e_Cauchy}, in the sense of Definition \ref{d2}. Moreover, since $\s\equiv 0$, $\cev{\pG}$
satisfies estimates \eqref{t_e4}, \eqref{t_e5} and \eqref{t_e6} with $\BIWi_{t,s}\equiv {\rm Id}$
and $\bvar_t^{s,\z}=\gY_{t-s}(\z)$ as in Definition \ref{ad1}.} 
Then, the function
\begin{equation}
 {f_t(z)}:={\int_{\R^{d+1}}}\cev{\pG}(t,z,T,\z)\phi(\z)d\z,\qquad (t,z)\in[0,T]\times \R^{d+1},
\end{equation}
solves problem \eqref{e_Cauchy2}. Since $\phi\in bC(\R^{d+1})$, we have
 $$\sup_{z\in\R^{d+1}}|f_z(z)|\le \|\phi\|_{\infty}\sup_{z\in\R^{d+1}}\int_{\R^{d+1}}\cev{\pG}(t,z,T,\z)d\z\le C$$
for a positive constant $C$. Estimate \eqref{coreq} is proven in \cite{MR2352998}, Proposition
3.3.
%
\endproof

\section{Backward It\^o calculus}\label{Itoback}
In this section we collect some basic result about {\it backward It\^o integrals} and the {\it
backward diffusion SPDE} (or {\it Krylov equation} according to \cite{MR3839316}). This is
standard material which resumes the original results in \cite{MR0339338}, \cite{MR653654},
\cite{MR673162}, \cite{MR665405}, \cite{MR706230} (see also the monographs \cite{MR3839316} and
\cite{MR1070361}).

Let $W=(W_{t})_{t\in[0,T]}$ be a $d$-dimensional Brownian motion on $(\O,\F,P,\F^{W})$ where
$\F^{W}$ denotes the standard Brownian filtration satisfying the usual assumptions. We consider
  $$\F^{W,t}_{T}=\s(\mathcal{G}_{t}\cup \mathcal{N}),\qquad \mathcal{G}_{t}=\s(W_{s}-W_{t},\, t\le s\le T),\qquad t\in[0,T],$$
the augmented  $\s$-algebra of Brownian increments between $t$ and $T$. Notice that
$(\F^{W,t}_{T})_{0\le t\le T}$ is a decreasing family of $\s$-algebras. Then the process
 $$\cev{W}_{t}:=W_{T}-W_{T-t},\qquad t\in[0,T],$$
is a Brownian motion on $(\O,\F,P,\cev{\F})$ where
  $$\cev{\F}_{t}:=\F^{W,T-t}_{T},\qquad t\in[0,T],$$
is the ``backward'' Brownian filtration. The backward stochastic It\^o integral is defined as
\begin{equation}\label{e2}
  \int_{t}^{s}u_{r}\star dW_{r}:=\int_{T-s}^{T-t}u_{T-r}d\cev{W}_{r},\qquad 0\le t\le s\le T,
\end{equation}
under the assumptions on $u$ for which the RHS of \eqref{e2} is defined in the usual It\^o sense,
that is
\begin{itemize}
  \item[i)] $t\mapsto u_{T-t}$ is $\cev{\F}$-progressively
  measurable (thus $u_{t}\in m \mathcal{F}^{W,t}_{T}$ for any $t\in[0,T]$);

  \item[ii)] $u\in L^{2}([0,T])$ a.s.
\end{itemize}
For practical purposes, if $u$ is continuous, the backward integral is the limit
\begin{equation}\label{e3}
  \int_{t}^{s}u_{r}\star dW_{r}:=\lim_{|\pi|\to 0^+}\sum_{k=1}^{n}u_{t_{k}}\left(W_{t_{k}}-W_{t_{k-1}}\right)
\end{equation}
in probability, where $\pi=\{t=t_{0}<t_{1}<\cdots<t_{n}=s\}$ denotes a partition of $[t,s]$.

A backward It\^o process is a process of the form
\begin{equation}\label{e4}
  X_{t}=X_{T}+\int_{t}^{T}b_{s}ds+\int_{t}^{T}\s_{s}\star dW_{s},\qquad t\in[0,T],
\end{equation}
also written in differential form as
\begin{equation}\label{e5}
  -d X_{t}=b_{t}dt+\s_{t}\star dW_{t}.
\end{equation}
\begin{theorem}[\bf Backward It\^o formula]\label{t1}
Let $v=v(t,x)\in C^{1,2}(\R_{\ge0}\times\R^{d})$ and let $X$ be the process in \eqref{e5}. Then
\begin{equation}\label{e6}
  -dv(t,X_{t})=\left((\p_{t}v)(t,X_{t})+\frac{1}{2}(\s_{t}\s_{t}^{\ast})_{ij}(\p_{x_{i}x_{j}}v)(t,X_{t})+(b_{t})_{i}(\p_{x_{i}}v)
  (t,X_{t})\right)dt+
  (\s_{t})_{ij}\left(\p_{x_{i}} v\right)(t,X_{t})\star dW^{j}_{t}.
\end{equation}
\end{theorem}
A crucial tool in our analysis is the following
\begin{theorem}[\bf Backward diffusion SPDE]\label{t2}
Assume $b,\s\in bC^{3}(\R_{\ge0}\times\R^{d})$ and denote by $s\mapsto X^{t,x}_{s}$ the solution
of the SDE
\begin{equation}\label{e7}
  dX^{t,x}_{s}=b(s,X^{t,x}_{s})ds+\s(s,X^{t,x}_{s})dW_{s}
\end{equation}
with initial condition $X^{t,x}_{t}=x$. Then the process $(t,x)\mapsto X^{t,x}_{T}$ solves the
backward SPDE
\begin{align}\label{e8}
  \begin{cases}
  -dX^{t,x}_{T}=\mathcal{L}X^{t,x}_{T}dt
 +\s_{ij}(t,x)\p_{x_{i}}X^{t,x}_{T}\star dW^{j}_{t}, \\
    X^{T,x}_{T}=x,
  \end{cases}
\end{align}
where
  \begin{equation}\label{e11}
  \mathcal{L}=\frac{1}{2}(\s(t,x)\s^{\ast}(t,x))_{ij}\p_{x_{j}x_{i}}+b_{i}(t,x)\p_{x_{i}}
\end{equation}
is the characteristic operator of $X$. More explicitly, in \eqref{e8} we have
  $$\mathcal{L}X^{t,x}_{T}\equiv \frac{1}{2}(\s(t,x)\s^{\ast}(t,x))_{ij}\p_{x_{j}x_{i}}X^{t,x}_{T}+b_{i}(t,x)\p_{x_{i}}X^{t,x}_{T}.$$
\end{theorem}
\begin{remark}
The regularity assumption of Theorem \ref{t2} on the coefficients is by no means optimal:
\cite{MR3839316}, Theorem 5.1, proves that $(t,x)\mapsto X^{t,x}_{T}$ is a generalized (or
classical, under non-degeneracy conditions) solution of \eqref{e8} if $b,\s\in
bC^{1}(\R_{\ge0}\times\R^{d})$.
\end{remark}
\proof For illustrative purposes we only consider the one-dimensional, autonomous case. A general
proof can be found in \cite{MR3839316}, Proposition 5.3. Here we follow the ``direct'' approach
proposed in \cite{MR706230}. By standard results for  stochastic flows (cf. \cite{MR1070361}),
$x\mapsto X^{t,x}_{T}$ is sufficiently regular to support the derivatives in the classical sense.
We use the Taylor expansion for $C^{2}$-functions:
\begin{equation}\label{e12}
  f(\d)-f(0)=\d f'(0)+\frac{\d^{2}}{2}f''(\l\d),\qquad \l\in[0,1].
\end{equation}
We have
\begin{align}
 X^{t,x}_{T}-x&=X^{t,x}_{T}-X^{T,x}_{T}\\
 &=\sum_{k=1}^{n}\left(X^{t_{k-1},x}_{T}-X^{t_{k},x}_{T}\right)
\intertext{(by the flow property)}
 &=\sum_{k=1}^{n}\left(X^{t_{k},X^{t_{k-1},x}_{t_{k}}}_{T}-X^{t_{k},x}_{T}\right)
\intertext{(by \eqref{e12} with $f(\d)=X_{T}^{t_{k},x+\d}$ and
 $\d=\Delta_{k}X:=X^{t_{k-1},x}_{t_{k}}-x$)}\label{e14}
 &=\sum_{k=1}^{n}\left(\Delta_{k}X\p_{x}X^{t_{k},x}_{T}+
 \frac{(\Delta_{k}X)^{2}}{2}\p_{xx}X^{t_{k},x+\l_{k}\Delta_{k}X}_{T}\right)
\end{align}
for some $\l_{k}=\l_{k}(\o)\in[0,1]$. Now, we have
  $$\Delta_{k}X=X^{t_{k-1},x}_{t_{k}}-x=\int_{t_{k-1}}^{t_{k}}b(X^{t_{k-1},x}_{s})ds+\int_{t_{k-1}}^{t_{k}}\s(X^{t_{k-1},x}_{s})dW_{s}.$$
Thus, setting
  $$\Delta_{k}t=t_{k}-t_{k-1},\qquad \Delta_{k}W=W_{t_{k}}-W_{t_{k-1}},\qquad \tilde{\Delta}_{k}X=b(x)\Delta_{k}t+\s(x)\Delta_{k}W,$$
by standard estimates for solutions of SDEs, we have
\begin{align}
 \Delta_{k}X-\tilde{\Delta}_{k}X=
 \int_{t_{k-1}}^{t_{k}}\left(b(X^{t_{k-1},x}_{s})-b(x)\right)ds+\int_{t_{k-1}}^{t_{k}}\left(\s(X^{t_{k-1},x}_{s})-\s(x)\right)dW_{s}
 &=\text{O}(\Delta_{k}t),\\ \label{e16}
 \p_{xx}X^{t_{k},x+\l_{k}\Delta_{k}X}_{T}-\p_{xx}X^{t_{k},x}_{T} &=\text{O}(\Delta_{k}t),
\end{align}
in the square mean sense or, more precisely,
  $$
  E\left[|\Delta_{k}X-\tilde{\Delta}_{k}X|^{2}+\left|\p_{xx}X^{t_{k},x+\l_{k}\Delta_{k}X}_{T}-
  \p_{xx}X^{t_{k},x}_{T}\right|^{2}\right]\le c(1+|x|^{2})
  (\Delta_{k}t)^{2}$$
with $c$ depending only on $T$ and the Lipschitz constants of $b,\s$. From \eqref{e14} we get
\begin{align}
 X^{t,x}_{T}-x
 &=\sum_{k=1}^{n}\left(\tilde{\Delta}_{k}X\p_{x}X^{t_{k},x}_{T}+
 \frac{(\tilde{\Delta}_{k}X)^{2}}{2}\p_{xx}X^{t_{k},x}_{T}\right)+\text{O}(\Delta_{k}t).
\end{align}
Next we recall {\eqref{e3}} and notice that
  $\p_{x}X^{t_{k},x}_{T},\p_{xx}X^{t_{k},x}_{T}\in m\F^{W,t_{k}}_{T}.$
Thus, passing to the limit, we have
\begin{align}
 \sum_{k=1}^{n}\tilde{\Delta}_{k}X\p_{x}X^{t_{k},x}_{T} &\longrightarrow
 \int_{t}^{T}b(t,x)\p_{x}X^{s,x}_{T}ds+\int_{t}^{T}\s(x)\p_{x}X^{s,x}_{T}\star dW_{s},\\
 \sum_{k=1}^{n}(\tilde{\Delta}_{k}X)^{2}\p_{xx}X^{t_{k},x}_{T} &\longrightarrow
 \int_{t}^{T}\s^{2}(x)\p_{xx}X^{s,x}_{T}ds,
\end{align}
in the square mean sense and this concludes the proof.
\endproof
We have a useful corollary of Theorem \ref{t2}.
\begin{corollary}[\bf Invariance of the backward diffusion SPDE]\label{c1}
For $v\in bC^{2}(\R^{d})$ and $X$ as in \eqref{e7}, let $V^{t,x}_{T}=v(X^{t,x}_{T})$. Then
$V^{t,x}_{T}$ satisfies the same SPDE \eqref{e8}, that is
\begin{equation}\label{e8b}
 -dV^{t,x}_{T}=\mathcal{L}V^{t,x}_{T}dt+
 \s_{ij}(t,x)\p_{x_{i}}V^{t,x}_{T}\star dW^{j}_{t}
\end{equation}
with terminal condition $V^{T,x}_{T}=g(x)$.
\end{corollary}
\begin{proof}
To fix ideas, we first consider the one-dimensional case: by the backward SPDE \eqref{e8} and the
backward It\^o formula \eqref{e6}, we have
\begin{align}
  -dv(X^{t,x}_{T})&=
  \left(\frac{\s^{2}(t,x)}{2}v''(X^{t,x}_{T})(\p_{x} X^{t,x}_{T})^{2}+\frac{\s^{2}(t,x)}{2}v'(X^{t,x}_{T})\p_{xx}X^{t,x}_{T}
  +b(t,x)v'(X^{t,x}_{T})\p_{x} X^{t,x}_{T}\right)dt\\
  &\quad+\s(t,x)v'(X^{t,x}_{T})\p_{x} X^{t,x}_{T}\star dW_{t}=
\intertext{(using the identities $\p_{x} V^{t,x}_{T}=v'(X^{t,x}_{T})\p_{x} X^{t,x}_{T}$ and
$\p_{xx}V^{t,x}_{T}=v''(X^{t,x}_{T})(\p_{x}X^{t,x}_{T})^{2}+
  v'(X^{t,x}_{T})\p_{xx}X^{t,x}_{T}$)}
  &=\left(\frac{\s^{2}(t,x)}{2}\p_{xx}V^{t,x}_{T}+b(t,x)\p_{x} V^{t,x}_{T}\right)dt+\s(t,x)\p_{x} V^{t,x}_{T}\star
  dW_{t}
\end{align}
and this proves the thesis. In general, we have
\begin{equation}\label{e9}
\begin{split}
  \p_{x_{h}} V^{t,x}_{T}&=(\nabla v)(X^{t,x}_{T})\p_{x_{h}} X^{t,x}_{T},\\
  \p_{x_{h}x_{k}}V^{t,x}_{T}&=
  (\p_{ij} v)(X^{t,x}_{T})(\p_{x_{h}}X^{t,x}_{T})_{i}(\p_{x_{k}}X^{t,x}_{T})_{j}+
  (\nabla v)(X^{t,x}_{T})(\p_{x_{h}x_{k}}X^{t,x}_{T}),
\end{split}
\end{equation}
and by \eqref{e8} and \eqref{e6}
\begin{align}
  -dv(X^{t,x}_{T})&=
  \left(\frac{1}{2}\left((\nabla X^{t,x}_{T})\s(t,x)((\nabla X^{t,x}_{T})\s(t,x))^{\ast}\right)_{ij}(\p_{ij}v)(X^{t,x}_{T})\right)dt\\
  &\quad+\left(\frac{1}{2}(\s(t,x)\s^{\ast}(t,x))_{ij}\p_{x_{j}x_{i}}X^{t,x}_{T}+b(t,x)\nabla X^{t,x}_{T}\right)(\nabla v)
  (X^{t,x}_{T})dt\\
  &\quad+(\nabla v)(X^{t,x}_{T})(\nabla X^{t,x}_{T})\s(t,x)\star dW_{t}=
\intertext{(by \eqref{e9})}
  &=\left(\frac{1}{2}(\s(t,x)\s^{\ast}(x))_{ij}\p_{x_{j}x_{i}}V^{t,x}_{T}+b(t,x)\nabla V^{t,x}_{T}\right)dt+\nabla V^{t,x}_{T}\s(t,x)\star
  dW_{t}.
\end{align}
\end{proof}

\section{Summary of notations}\label{secnot}

The points in $\R^2$ are denoted by $z=(x,v)$ and $\z=(\x,\nu)$, with $z$ generally standing for
the \textit{initial point} and $\z$ for the \textit{final point}. Analogously the points in
$\R^{d+1}$ are denoted by $z=(x,v_{1},\dots,v_{d})$ and $\z=(\x,\n_{1},\dots,\n_{d})$. Moreover,
as a general rule, when a quantity depends on both an \textit{initial state} and a \textit{final
state}, the variables which describe the initial state are always appended first, regardless of
whether they act as the \textit{pole} or not: in particular this is the case when denoting
deterministic or stochastic flows, conditioned or unconditioned densities, deterministic or
stochastic fundamental solutions. In particular, as in Section \ref{sec1} we denote by:
\begin{itemize}
\item[•] $t\mapsto\gY_t(z)=(x+tv,v)$ the integral curve, starting from $z$, of the advection vector field $v\p_x$;
\item[•] $ \FIW(x,v):= (x,\gIW_{t,s}(x,v))$ is the forward stochastic flow of diffeomorphism defined by the SDE
$$ \gIW_{t,s}(x,v)=v-\int_{t}^s\s_\t(x,\gIW_{t,\t}(x,v))dW_\t, \qquad s\in [t,T],$$
and $\FIWi$ is its inverse;
\item[•] $\BIW(x,v):= \left(x,\bIW_{t,s}(x,v)\right)$ is the backward stochastic flow of diffeomorphism defined by
$$\bIW_{t,s}(x,v)=v+\int_{t}^{s}\s_\t(x,\bIW_{\t,s}(x,v))\star dW_\t,\qquad t\in[0,s],$$
and $\BIW(x,v)$ is its inverse;
\item[•] $\gvar_{s}^{t,z}$ is the integral curve, starting at $z$ at time $t$, defined by the ODE
$$\gvar_{s}^{t,z}=z+\int_{t}^{s}\Qg_{t,\t}(\gvar_{\t}^{t,z})d\t,\qquad s\in[t,T],$$
where $\Qg_{t,s}(z):=\Big((\gIW_{t,s})_1(z),-(\gIW_{t,s}(z))_1(\nabla_v\gIW_{t,s})^{-1}(z){\p_x\gIW_{t,s}(z)}\Big)$;
\item[•] $\bvar_{t}^{s,\z}$ is the integral curve, ending at $\z$ at time $s$ defined by
$$\bvar_{t}^{s,\z}=\z+\int_{t}^{s}\Qgb_{\t,s}(\bvar_{\t}^{s,\z})d\t,\qquad t\in[0,s],$$ where
{$\Qgb_{t,s}(z):=\Big((\bIW_{t,s})_1(z),-(\bIW_{t,s}(z))_1(\nabla_v\bIW_{t,s})^{-1}(z){\p_x\bIW_{t,s}(z)}\Big)$.}
\end{itemize}
Lastly, $\G_{\l}(t,x,v)$ denotes the Gaussian kernel
\begin{equation}
 \Gamma_{\l}(t,x,v)=\frac{1}{t^{\frac{d+3}{2}}}\exp\left(-\frac{1}{2\l}\left(\frac{x^{2}}{t^{3}}+\frac{|v|^{2}}{t}\right)\right),
  \qquad t>0,\ (x,v)\in\R\times\R^{d},\ \l>0.
\end{equation}

\bigskip\noindent
{\bf Acknowledgments.} This was supported by the Gruppo Nazionale per l'Analisi Matematica, la
Probabilit\`a e le loro Applicazioni (GNAMPA) of the Istituto Nazionale di Alta Matematica
(INdAM).


\bibliographystyle{spmpsci}      

\bibliography{bib1}

\def\cprime{$'$} \def\cprime{$'$} \def\cprime{$'$}
  \def\lfhook#1{\setbox0=\hbox{#1}{\ooalign{\hidewidth
  \lower1.5ex\hbox{'}\hidewidth\crcr\unhbox0}}} \def\cprime{$'$}
  \def\cprime{$'$} \def\cprime{$'$} \def\cprime{$'$} \def\cprime{$'$}
  \def\polhk#1{\setbox0=\hbox{#1}{\ooalign{\hidewidth
  \lower1.5ex\hbox{`}\hidewidth\crcr\unhbox0}}}
\begin{thebibliography}{10}
\providecommand{\url}[1]{{#1}}
\providecommand{\urlprefix}{URL }
\expandafter\ifx\csname urlstyle\endcsname\relax
  \providecommand{\doi}[1]{DOI~\discretionary{}{}{}#1}\else
  \providecommand{\doi}{DOI~\discretionary{}{}{}\begingroup
  \urlstyle{rm}\Url}\fi

\bibitem{anceschi}
Anceschi, F., Polidoro, S.: A survey on the classical theory for {K}olmogorov
  equation.
\newblock Matematiche (Catania) \textbf{75}(1), 221--258 (2020).
\newblock \doi{10.4418/2020.75.1.11}.
\newblock \urlprefix\url{https://doi-org.ezproxy.unibo.it/10.4418/2020.75.1.11}

\bibitem{BarucciPolidoroVespri}
Barucci, E., Polidoro, S., Vespri, V.: Some results on partial differential
  equations and {A}sian options.
\newblock Math. Models Methods Appl. Sci. \textbf{11}(3), 475--497 (2001)

\bibitem{brampol}
Bramanti, M., Polidoro, S.: Fundamental solutions for
  {K}olmogorov-{F}okker-{P}lanck operators with time-depending measurable
  coefficients.
\newblock Math. Eng. \textbf{2}(4), 734--771 (2020).
\newblock \doi{10.3934/mine.2020035}.
\newblock \urlprefix\url{https://doi-org.ezproxy.unibo.it/10.3934/mine.2020035}

\bibitem{Cercignani}
Cercignani, C.: The {B}oltzmann equation and its applications.
\newblock Springer-Verlag, New York (1988)

\bibitem{MR736147}
Chaleyat-Maurel, M., Michel, D.: Hypoellipticity theorems and conditional laws.
\newblock Z. Wahrsch. Verw. Gebiete \textbf{65}(4), 573--597 (1984).
\newblock \doi{10.1007/BF00531840}.
\newblock \urlprefix\url{https://doi.org/10.1007/BF00531840}

\bibitem{Chow94}
Chow, P.L., Jiang, J.L.: Stochastic partial differential equations in
  {H}\"older spaces.
\newblock Probab. Theory Related Fields \textbf{99}(1), 1--27 (1994).
\newblock \doi{10.1007/BF01199588}.
\newblock \urlprefix\url{http://dx.doi.org/10.1007/BF01199588}

\bibitem{MR2659772}
Delarue, F., Menozzi, S.: Density estimates for a random noise propagating
  through a chain of differential equations.
\newblock J. Funct. Anal. \textbf{259}(6), 1577--1630 (2010).
\newblock \doi{10.1016/j.jfa.2010.05.002}.
\newblock \urlprefix\url{https://doi.org/10.1016/j.jfa.2010.05.002}

\bibitem{Desvillettes}
Desvillettes, L., Villani, C.: On the trend to global equilibrium in spatially
  inhomogeneous entropy-dissipating systems: the linear {F}okker-{P}lanck
  equation.
\newblock Comm. Pure Appl. Math. \textbf{54}(1), 1--42 (2001)

\bibitem{MR2352998}
Di~Francesco, M., Pascucci, A.: A continuous dependence result for
  ultraparabolic equations in option pricing.
\newblock J. Math. Anal. Appl. \textbf{336}(2), 1026--1041 (2007).
\newblock \doi{10.1016/j.jmaa.2007.03.031}.
\newblock
  \urlprefix\url{https://doi-org.ezproxy.unibo.it/10.1016/j.jmaa.2007.03.031}

\bibitem{MR657581}
Folland, G.B., Stein, E.M.: Hardy spaces on homogeneous groups,
  \emph{Mathematical Notes}, vol.~28.
\newblock Princeton University Press, Princeton, N.J.; University of Tokyo
  Press, Tokyo (1982)

\bibitem{MR2130405}
Helffer, B., Nier, F.: Hypoelliptic estimates and spectral theory for
  {F}okker-{P}lanck operators and {W}itten {L}aplacians, \emph{Lecture Notes in
  Mathematics}, vol. 1862.
\newblock Springer-Verlag, Berlin (2005).
\newblock \doi{10.1007/b104762}.
\newblock \urlprefix\url{https://doi-org.ezproxy.unibo.it/10.1007/b104762}

\bibitem{Hormander}
H{\"o}rmander, L.: Hypoelliptic second order differential equations.
\newblock Acta Math. \textbf{119}, 147--171 (1967)

\bibitem{MR583435}
Kallianpur, G.: Stochastic filtering theory, \emph{Applications of
  Mathematics}, vol.~13.
\newblock Springer-Verlag, New York-Berlin (1980)

\bibitem{Kolmogorov2}
Kolmogorov, A.: {Zufallige Bewegungen. (Zur Theorie der Brownschen Bewegung.).}
\newblock Ann. of Math., II. Ser. \textbf{35}, 116--117 (1934)

\bibitem{MR0339338}
Krylov, N.V.: The selection of a {M}arkov process from a {M}arkov system of
  processes, and the construction of quasidiffusion processes.
\newblock Izv. Akad. Nauk SSSR Ser. Mat. \textbf{37}, 691--708 (1973)

\bibitem{Krylov17}
Krylov, N.V.: H\"ormander's theorem for stochastic partial differential
  equations.
\newblock Algebra i Analiz \textbf{27}(3), 157--182 (2015)

\bibitem{MR653654}
Krylov, N.V., Rozovsky, B.L.: On the first integrals and {L}iouville equations
  for diffusion processes.
\newblock In: Stochastic differential systems ({V}isegr\'{a}d, 1980),
  \emph{Lecture Notes in Control and Information Sci.}, vol.~36, pp. 117--125.
  Springer, Berlin-New York (1981)

\bibitem{MR673162}
Krylov, N.V., Rozovsky, B.L.: Characteristics of second-order degenerate
  parabolic {I}t\^{o} equations.
\newblock Trudy Sem. Petrovsk. (8), 153--168 (1982)

\bibitem{MR1795614}
Krylov, N.V., Zatezalo, A.: A direct approach to deriving filtering equations
  for diffusion processes.
\newblock Appl. Math. Optim. \textbf{42}(3), 315--332 (2000).
\newblock \doi{10.1007/s002450010015}.
\newblock \urlprefix\url{http://dx.doi.org/10.1007/s002450010015}

\bibitem{MR665405}
Kunita, H.: On backward stochastic differential equations.
\newblock Stochastics \textbf{6}(3-4), 293--313 (1981/82).
\newblock \doi{10.1080/17442508208833209}.
\newblock
  \urlprefix\url{https://doi-org.ezproxy.unibo.it/10.1080/17442508208833209}

\bibitem{MR705933}
Kunita, H.: Stochastic partial differential equations connected with nonlinear
  filtering.
\newblock In: Nonlinear filtering and stochastic control ({C}ortona, 1981),
  \emph{Lecture Notes in Math.}, vol. 972, pp. 100--169. Springer, Berlin
  (1982).
\newblock \doi{10.1007/BFb0064861}.
\newblock \urlprefix\url{https://doi.org/10.1007/BFb0064861}

\bibitem{MR1070361}
Kunita, H.: Stochastic flows and stochastic differential equations,
  \emph{Cambridge Studies in Advanced Mathematics}, vol.~24.
\newblock Cambridge University Press, Cambridge (1990)

\bibitem{Lions1}
Lions, P.L.: On {B}oltzmann and {L}andau equations.
\newblock Philos. Trans. Roy. Soc. London Ser. A \textbf{346}(1679), 191--204
  (1994)

\bibitem{MR1755998}
Mikulevicius, R.: On the {C}auchy problem for parabolic {SPDE}s in {H}\"{o}lder
  classes.
\newblock Ann. Probab. \textbf{28}(1), 74--103 (2000).
\newblock \doi{10.1214/aop/1019160112}.
\newblock
  \urlprefix\url{https://doi-org.ezproxy.unibo.it/10.1214/aop/1019160112}

\bibitem{MR553909}
Pardoux, E.: Stochastic partial differential equations and filtering of
  diffusion processes.
\newblock Stochastics \textbf{3}(2), 127--167 (1979).
\newblock \doi{10.1080/17442507908833142}.
\newblock \urlprefix\url{https://doi.org/10.1080/17442507908833142}

\bibitem{Pascucci2011}
Pascucci, A.: {PDE} and martingale methods in option pricing.
\newblock Bocconi\&Springer Series. Springer-Verlag, New York (2011)

\bibitem{pasc:pesc:19}
Pascucci, A., Pesce, A.: On stochastic {L}angevin and {F}okker-{P}lanck
  equations: the two-dimensional case.
\newblock arXiv:1910.05301  (2019)

\bibitem{PascucciPesce1}
Pascucci, A., Pesce, A.: The parametrix method for parabolic {SPDE}s.
\newblock Stochastic Process. Appl. \textbf{130}(10), 6226--6245 (2020).
\newblock \doi{10.1016/j.spa.2020.05.008}.
\newblock
  \urlprefix\url{https://doi-org.ezproxy.unibo.it/10.1016/j.spa.2020.05.008}

\bibitem{MR3706782}
Qiu, J.: H\"{o}rmander-type theorem for {I}t\^{o} processes and related
  backward {SPDE}s.
\newblock Bernoulli \textbf{24}(2), 956--970 (2018).
\newblock \doi{10.3150/16-BEJ816}.
\newblock \urlprefix\url{https://doi.org/10.3150/16-BEJ816}

\bibitem{MR3839316}
Rozovsky, B.L., Lototsky, S.V.: Stochastic evolution systems, \emph{Probability
  Theory and Stochastic Modelling}, vol.~89.
\newblock Springer, Cham (2018).
\newblock \doi{10.1007/978-3-319-94893-5}.
\newblock
  \urlprefix\url{https://doi-org.ezproxy.unibo.it/10.1007/978-3-319-94893-5}.
\newblock Linear theory and applications to non-linear filtering, Second
  edition of [ MR1135324]

\bibitem{MR706230}
Veretennikov, A.Y.: ``{I}nverse diffusion'' and direct derivation of stochastic
  {L}iouville equations.
\newblock Mat. Zametki \textbf{33}(5), 773--779 (1983)

\bibitem{MR1352749}
Veretennikov, A.Y.: On backward filtering equations for {SDE} systems (direct
  approach).
\newblock In: Stochastic partial differential equations ({E}dinburgh, 1994),
  \emph{London Math. Soc. Lecture Note Ser.}, vol. 216, pp. 304--311. Cambridge
  Univ. Press, Cambridge (1995).
\newblock \doi{10.1017/CBO9780511526213.019}.
\newblock
  \urlprefix\url{https://doi-org.ezproxy.unibo.it/10.1017/CBO9780511526213.019}

\bibitem{Veretennikov}
Veretennikov, A.Y.: On {SPDE} and backward filtering equations for {SDE}
  systems (direct approach).
\newblock https://arxiv.org/abs/1607.00333  (July 2016)

\end{thebibliography}
\end{document}